\documentclass[12pt]{article}
\usepackage{multicol}
\usepackage{amsmath, amssymb, amscd, enumerate, amsfonts, latexsym, amsxtra, 
amsthm, bm, multirow}
\usepackage{indentfirst}
\newtheorem{Theorem}{Theorem}[section]
\newtheorem{Proposition}[Theorem]{Proposition}
\newtheorem{Lemma}[Theorem]{Lemma}

\begin{document}
\allowdisplaybreaks[3]
\centerline{\Large Primitive idempotents of the hyperalgebra for the $r$-th} \vspace{3mm}
\centerline{\Large Frobenius kernel of ${\rm SL}(2,k)$} \vspace{7mm}
\centerline{Yutaka Yoshii 
\footnote{ E-mail address: yutaka.yoshii.6174@vc.ibaraki.ac.jp}}  \vspace{5mm}
\centerline{College of Education,   
Ibaraki University,}
\centerline{2-1-1 Bunkyo, Mito, Ibaraki, 310-8512, Japan}
\begin{abstract}
In this paper we construct  primitive idempotents of  
the hyperalgebra for the $r$-th Frobenius kernel of the algebraic group ${\rm SL}(2,k)$. 
\end{abstract}
{\itshape Key words:} Primitive idempotents, Hyperalgebras, Algebraic groups, Frobenius kernels
\\
{\itshape Mathematics Subject Classification 2010:} 16S30; 17B45; 20G05

\section{Introduction}
 Let $k$ be an algebraically closed field of characteristic $p>0$. Let $G$ be a connected, simply connected and semisimple algebraic group over $k$ which is split over the finite field $\mathbb{F}_p$ of order $p$. 

The representation theory of $G$ is closely related to that of   
the $r$-th Frobenius kernel $G_r$.  For example, 
all simple modules for $G_r$ can be lifted to some 
simple modules for $G$.  In addition, the  projective indecomposable modules (PIMs) for $G_r$ can be 
lifted to some modules for $G$ if $p$ is not too small.  
So it is important to construct PIMs for $G_r$ 
in order to know the structure of some important  $G$-modules. 

On the other hand,  the representation theory of $G$ can be identified with the 
locally finite representation theory of the corresponding 
(infinite-dimensional) $k$-algebra $\mathcal{U}$ which is called the hyperalgebra of $G$, and the representation theory of 
$G_r$ can be identified with that of the corresponding finite-dimensional 
hyperalgebra $\mathcal{U}_r$.  If we have a decomposition of the unity 
$1 \in \mathcal{U}_r$ into a sum of pairwise orthogonal primitive idempotents, we can construct the PIMs for $\mathcal{U}_r$. Unfortunately, it seems that such a decomposition 
where each primitive idempotent is explicitly described is not known except 
when $G$ is of type $A_1$ (i.e. $G={\rm SL}(2,k)$) and $r=1$, in this case the decomposition is given in  Seligman's paper \cite{seligman03}.     

In this paper we shall generalize Seligman's result. More concretely, when 
$G={\rm SL}(2,k)$, we construct pairwise orthogonal primitive idempotents  
whose sum is the unity $1$ explicitly in 
$\mathcal{U}_r$ for an arbitrary positive integer $r$.

The main result is given in Section 5. In Section 4 
we first construct the primitive idempotents 
in $\mathcal{U}_1$ using Seligman's method.
Next, in Section 5 we construct primitive idempotents in $\mathcal{U}_r$ for $r \geq 2$ using 
the primitive idempotents in $\mathcal{U}_1$ given in the previous section. There we use 
a linear map constructed in \cite{gros-kaneda11}, which 'splits' the Frobenius map on 
$\mathcal{U}$.  Finally we also determine the PIMs for 
$\mathcal{U}_r$ which the idempotents generate.    For convenience, 
this result will be given more generally for a larger algebra $\mathcal{U}_{r,r'}$ 
$(r'>r)$ in Theorem 5.8. 

\section{Preliminaries}
From now on, let $G={\rm SL}_2(k)$. Let 
$$X= \left( \begin{array}{cc}
         0 & 1 \\
         0 & 0 \end{array} \right),\ \ \ 
Y= \left( \begin{array}{cc}
         0 & 0 \\
         1 & 0 \end{array} \right),\ \ \ 
H= \left( \begin{array}{cc}
         1 & 0 \\
         0 & {-1} \end{array} \right)$$
be the standard basis in the simple complex Lie algebra $\mathfrak{g}_{\mathbb{C}}=\mathfrak{sl}_2(\mathbb{C})$. We define a subring 
$\mathcal{U}_{\mathbb{Z}}$ of the universal enveloping algebra $\mathcal{U}_{\mathbb{C}}$ of  $\mathfrak{g}_{\mathbb{C}}$  
generated by $X^{(m)}=X^m/m!$ and  $Y^{(m)}=Y^m/m!$ with $m \in \mathbb{Z}_{\geq 0}$. 
For  $z=H$ or $-H$ in $\mathcal{U}_{\mathbb{Z}}$, set 
$${z+c \choose m }= \dfrac{(z+c)(z+c-1) \cdots (z+c-m+1)}{m!}$$
for $c \in \mathbb{Z}$ and $m \in \mathbb{Z}_{\geq 0}$, which also lies in 
$\mathcal{U}_{\mathbb{Z}}$. Then the elements 
$$Y^{(m)} {H \choose n} X^{(m')}$$
with $m, m', n \in \mathbb{Z}_{\geq 0}$ form a $\mathbb{Z}$-basis of $\mathcal{U}_{\mathbb{Z}}$. The $k$-algebra 
$\mathcal{U}_{\mathbb{Z}} \otimes_{\mathbb{Z}} k$ is denoted by $\mathcal{U}$ and called 
the hyperalgebra of $G$, which is equipped with a structure of Hopf algebra over $k$. 
We  use the same notation for the images in $\mathcal{U}$ 
of the elements 
in $\mathcal{U}_{\mathbb{Z}}$. In the following propositions we shall give some 
well-known formulas, which are repeatedly used to carry out 
calculations in $\mathcal{U}$ (for example, see \cite[\S3]{gros12}, \cite[\S4]{gros-kaneda15}, 
\cite[\S5]{haboush80}, \cite[\S26]{humphreysbook}, \cite[I. ch. 7]{jantzenbook}).

\begin{Proposition}
For $m,n \in \mathbb{Z}_{\geq 0}$ and  
$s,t \in \mathbb{Z}$, the following holds in  $\mathcal{U}$: \\
{\rm (i)} $\displaystyle{X^{(m)} Y^{(n)} = \sum_{i=0}^{{\rm min}(m,n)} Y^{(n-i)} 
{H-m-n+2i \choose i} X^{(m-i)}}$,  \\ \ \ \ \ 
 $\displaystyle{Y^{(m)} X^{(n)} = \sum_{i=0}^{{\rm min}(m,n)} X^{(n-i)} 
{-H-m-n+2i \choose i} Y^{(m-i)}}$, \\
{\rm (ii)} $\displaystyle{{H +s \choose m} X^{(n)}= X^{(n)} {H+s +2n \choose m}}$, \  $\displaystyle{{H +s \choose m} Y^{(n)}= Y^{(n)} {H+s -2n \choose m}}$, \\
{\rm (iii)} $\displaystyle{X^{(m)} X^{(n)} = {m+n \choose n} X^{(m+n)}}$, \  
 $\displaystyle{Y^{(m)} Y^{(n)} = {m+n \choose n} Y^{(m+n)}}$, \\
{\rm (iv)} $\displaystyle{{H \choose m} {H \choose n}= \sum_{i=0}^{{\rm min}(m,n) }
{n+m-i \choose m} {m \choose i} {H \choose n+m-i}}$, \\
{\rm (v)} $\displaystyle{{H+s+t \choose m} = 
\sum_{i=0}^{m} {s \choose m-i} {H+t \choose i}}$.
\end{Proposition}

\begin{Proposition}
Let $m, m' , n, n'  \in \mathbb{Z}$ with $n' \geq 0$ and $0 \leq m,n \leq p-1$. Then 
$${m+m'p \choose n+n'p} \equiv {m \choose n} {m' \choose n'}\ ({\rm mod}\ p).$$
\end{Proposition}
\ \\ 

Let 
$${\rm Fr}:G \longrightarrow G,\ \ \ 
\left( \begin{array}{cc}
         a & b \\
         c & d \end{array} \right)
\mapsto 
\left( \begin{array}{cc}
         {a^p} & {b^p} \\
         {c^p} & {d^p} \end{array} \right)$$
be the geometric Frobenius map. This map induces the $k$-algebra endomorphism 
${\rm Fr}: \mathcal{U} \rightarrow \mathcal{U}$ which is defined by 
$${\rm Fr}(X^{(m)}) = 
\left\{ \begin{array}{cl}
         {X^{(m/p)}} & {\mbox{if $p \mid m$,}} \\
         0 & {\mbox{if $p \nmid m$}} \end{array} \right. 
\ \ \ \ \ \ \mbox{and  \ }
{\rm Fr}(Y^{(m)}) = 
\left\{ \begin{array}{cl}
         {Y^{(m/p)}} & {\mbox{if $p \mid m$,}} \\
         0 & {\mbox{if $p \nmid m$}} \end{array} \right. .$$
Then we also have 
$${\rm Fr} \bigg( {H \choose m} \bigg) = 
\left\{ \begin{array}{cl}
         {{H \choose m/p}} & {\mbox{if $p \mid m$,}} \\
         0 & {\mbox{if $p \nmid m$}} \end{array} \right. .$$

Let $\mathcal{U}^+$ (resp. $\mathcal{U}^-$) be the subalgebra of $\mathcal{U}$ 
generated by $X^{(p^i)}$ (resp.  $Y^{(p^i)}$) with $i \in \mathbb{Z}_{\geq 0}$ 
, and let $\mathcal{U}^0$ be 
the subalgebra of $\mathcal{U}$ generated by ${H \choose p^i}$ with 
$i \in \mathbb{Z}_{\geq 0}$. The elements 
$Y^{(m)} {H \choose n} X^{(m')}$ 
with $m, m', n \in \mathbb{Z}_{\geq 0}$ form a $k$-basis of $\mathcal{U}$ and 
we have a triangular decomposition 
$\mathcal{U} = \mathcal{U}^- \mathcal{U}^0 \mathcal{U}^+$. We say that an  element 
$z \in \mathcal{U}$ has degree $d$ if it is a $k$-linear combination of  the elements 
$Y^{(m)} {H \choose n} X^{(m')}$ 
with $m, m', n \in \mathbb{Z}_{\geq 0}$ and $m'-m =d$.
For a positive integer 
$r \in \mathbb{Z}_{> 0}$, let $\mathcal{U}_r$ be the subalgebra of $\mathcal{U}$ generated by 
$X^{(p^i)}$ and $Y^{(p^i)}$ with $0 \leq i  \leq r-1$. 
This is a finite-dimensional algebra of dimension $p^{3r}$ which has 
$Y^{(m)} {H \choose n} X^{(m')}$ 
with $0 \leq m, m', n \leq p^r-1$ as a basis, 
and it can be identified with the hyperalgebra of the $r$-th Frobenius kernel 
$G_r = {\rm Ker} ({\rm Fr}^r)$ of $G$. 
Let $\mathcal{U}_r^+$ (resp. $\mathcal{U}_r^-$) be the subalgebra of $\mathcal{U}$ 
generated by $X^{(p^i)}$ (resp.  $Y^{(p^i)}$) with $0 \leq i \leq r-1$ 
, and $\mathcal{U}_r^0$ 
the subalgebra of $\mathcal{U}$ generated by ${H \choose p^i}$ with 
$0 \leq i  \leq r-1$. Then we define four subalgebras 
$\mathcal{U}^{\geq 0}, \mathcal{U}^{\leq 0}, \mathcal{U}_r^{\geq 0}$ and 
$\mathcal{U}_r^{\leq 0}$ as
$$\mathcal{U}^{\geq 0}= \mathcal{U}^0\mathcal{U}^+, \ \ \ 
\mathcal{U}^{\leq 0}= \mathcal{U}^-\mathcal{U}^0,\ \ \ 
\mathcal{U}_r^{\geq 0}= \mathcal{U}_r^0 \mathcal{U}_r^+\ \ \mbox{and}\ \  
\mathcal{U}_r^{\leq 0}= \mathcal{U}_r^- \mathcal{U}_r^0.
$$
In dealing with the subalgebra $\mathcal{U}_r$ it is often convenient to   consider the larger  subalgebras 
$\widetilde{\mathcal{U}}_r= \mathcal{U}_r^- \mathcal{U}^0 \mathcal{U}_r^+$ generated by 
$X^{(p^i)}$, $Y^{(p^i)}$ and ${H \choose p^l}$ with $0 \leq i \leq r-1$ and 
$l \in \mathbb{Z}_{\geq 0}$, and 
$\mathcal{U}_{r,r'}= \mathcal{U}_r^- \mathcal{U}_{r'}^0 \mathcal{U}_r^+$ generated by 
$X^{(p^i)}$, $Y^{(p^i)}$ and ${H \choose p^l}$ with $0 \leq i \leq r-1$ and 
$0 \leq l \leq r'-1$ for $r' > r$. Then the elements 
$Y^{(m)} {H \choose n} X^{(m')}$ 
with $0 \leq m, m' \leq p^r-1$ and  $n \in \mathbb{Z}_{\geq 0}$ (resp. $0 \leq n \leq p^{r'}-1$) form a $k$-basis of $\widetilde{\mathcal{U}}_r$ (resp. $\mathcal{U}_{r,r'}$).

Consider the $k$-linear map ${\rm Fr}': \mathcal{U} \rightarrow \mathcal{U}$ defined 
by 
$$Y^{(m)} {H \choose n} X^{(m')} \mapsto 
Y^{(mp)} {H \choose np} X^{(m'p)}.$$
Clearly we have ${\rm Fr} \circ {\rm Fr}' = id_{\mathcal{U}}$. 
Let ${\rm Fr}'^+$ (resp. ${\rm Fr}'^-, {\rm Fr}'^0, {\rm Fr}'^{\geq 0}$, ${\rm Fr}'^{\leq 0}$) be the restriction of ${\rm Fr}'$ to $\mathcal{U}^+$ 
(resp. $\mathcal{U}^-, \mathcal{U}^0, \mathcal{U}^{\geq 0}$, $\mathcal{U}^{\leq 0}$).   
These five restriction maps are   homomorphisms of $k$-algebras, whereas  
${\rm Fr}'$ is not (see \cite[1.2]{gros-kaneda11}). 

In this paper, the symbol $\otimes$ is assumed to  mean a tensor product over $k$. 

\begin{Proposition}
For  $n,n' \in \mathbb{Z}_{> 0}$ with $n' > n$, the multiplication map 
$$\mathcal{U}_n \otimes {\rm Fr}'^n(\mathcal{U}_{n'-n}) \rightarrow \mathcal{U}_{n'}$$ 
is a $k$-linear isomorphism. 
\end{Proposition}

\begin{proof}
The linearity is clear. Since both the $k$-vector spaces 
$\mathcal{U}_n \otimes {\rm Fr}'^n(\mathcal{U}_{n'-n})$ 
and $\mathcal{U}_{n'}$ have dimension $p^{3n'}$, it is enough to show that the map is injective. Let $z$ be a nonzero element of 
$\mathcal{U}_n \otimes {\rm Fr}'^n(\mathcal{U}_{n'-n})$. Note that 
$$Y^{(m_1)} {H \choose n_1} X^{(m_1')} \otimes 
Y^{(m_2p^n)} {H \choose n_2p^n} X^{(m_2'p^n)}$$
with $0 \leq m_1,n_1,m_1' \leq p^n-1$ and $0 \leq m_2,n_2,m_2' \leq p^{n'-n}-1$ form a 
$k$-basis of $\mathcal{U}_n \otimes {\rm Fr}'^n(\mathcal{U}_{n'-n})$. Moreover, by 
using the formulas in Propositions 2.1 and 2.2 we easily see that such a basis vector is  
mapped to  
$$Y^{(m_1)} {H \choose n_1} X^{(m_1')} 
Y^{(m_2p^n)} {H \choose n_2p^n} X^{(m_2'p^n)} 
= Y^{(m_1+m_2p^n)}{H \choose n_1+n_2p^n} X^{(m_1'+m_2'p^n)} + u,$$
where $u$ is a $k$-linear combination of some basis vectors  of the form 
$Y^{(m_3)} {H \choose n_3} X^{(m_3')} $
with $0 \leq m_3,n_3,m_3' \leq p^{n'}-1$ and 
$$m_3+n_3+m_3' < m_1+n_1+m_1'+(m_2+n_2+m_2')p^n.$$ Therefore, if we take a $6$-tuple 
$(m_1,n_1,m_1',m_2,n_2,m_2')$ where $$m_1+n_1+m_1'+(m_2+n_2+m_2')p^n$$ is 
the largest integer with the coefficient of 
 $Y^{(m_1)} {H \choose n_1} X^{(m_1')} \otimes 
Y^{(m_2p^n)} {H \choose n_2p^n} X^{(m_2'p^n)}$ in the expression of $z$ as a $k$-linear combination of the basis vectors of  
$\mathcal{U}_n \otimes {\rm Fr}'^n(\mathcal{U}_{n'-n})$ being nonzero, 
then the coefficient of  
$Y^{(m_1+m_2p^n)} {H \choose n_1+n_2p^n} X^{(m_1'+m_2'p^n)}$ in the expression of 
the image of $z$ as a $k$-linear combination of the basis vectors of 
$\mathcal{U}_{n'}$ is also nonzero. This shows that the multiplication map is injective, and 
the proposition follows. 
\end{proof}
\noindent {\bf Remark.} These multiplication maps $\mathcal{U}_n \otimes {\rm Fr}'^n(\mathcal{U}_{n'-n}) \rightarrow \mathcal{U}_{n'}$ ($n'>n>0$) induce the multiplication maps 
$\mathcal{U}_n \otimes {\rm Fr}'^n(\mathcal{U}) \rightarrow \mathcal{U}$ and 
$\bigotimes_{i \geq 0} {\rm Fr}'^i(\mathcal{U}_1) \rightarrow \mathcal{U}$ (for the natural 
ordering of $\mathbb{Z}_{\geq 0}$), all of which are 
$k$-linear isomorphisms. \\ \\

For later use, let  $\mathcal{A}$ be the subalgebra of $\mathcal{U}$ which is generated by 
$\mathcal{U}^0$ and the elements $Y^{(p^i)}X^{(p^i)}$ with $i \in \mathbb{Z}_{\geq 0}$. 

\begin{Proposition} The following holds. \\ \\
{\rm (i)} $\mathcal{A}$ is commutative. \\ \\
{\rm (ii)} $\mathcal{A}$ is the centralizer of $\mathcal{U}^0$ in $\mathcal{U}$, consisting 
of all the elements of degree $0$ in $\mathcal{U}$. \\ \\
{\rm (iii)} $\mathcal{A}$ is free over $\mathcal{U}^0$ of basis $X^{(m)} Y^{(m)}$, 
$m \in \mathbb{Z}_{\geq 0}$, and also of basis $Y^{(m)}X^{(m)}$, 
$m \in \mathbb{Z}_{\geq 0}$. 
\end{Proposition}
\begin{proof}
(i) Let $s,t \in \mathbb{Z}_{\geq 0}$.  Since  
\begin{eqnarray*}
{H \choose p^s} Y^{(p^t)}X^{(p^t)} 
& = & Y^{(p^t)}{H-2p^t \choose p^s}X^{(p^t)} \\
& = & Y^{(p^t)} \sum_{i=0}^{p^s} {-2p^t \choose i } {H \choose p^s-i} X^{(p^t)} \\
& = & Y^{(p^t)}X^{(p^t)} \sum_{i=0}^{p^s} {-2p^t \choose i } {H+2p^t \choose p^s-i} \\
& = & Y^{(p^t)}X^{(p^t)} {H \choose p^s},
\end{eqnarray*} 
each $Y^{(p^t)}X^{(p^t)} $ commutes with all elements of $\mathcal{U}^0$. 
On the other hand, we have \\
\begin{eqnarray*}
\lefteqn{Y^{(p^s)}X^{(p^s)}Y^{(p^t)}X^{(p^t)}} \\
& = & Y^{(p^s)} \sum_{i=0}^{{\rm min}(p^s, p^t)} Y^{(p^t-i)} 
{H-p^s-p^t+2i \choose i} X^{(p^s-i)} X^{(p^t)} \\
& = & \sum_{i=0}^{{\rm min}(p^s, p^t)} {p^s+p^t-i \choose p^s} {p^s+p^t-i \choose p^t} 
Y^{(p^s+p^t-i)} 
{H-p^s-p^t+2i \choose i} X^{(p^s+p^t-i)},  
\end{eqnarray*}
which is symmetric with respect to  $p^s$ and $p^t$. Therefore, 
$Y^{(p^t)}X^{(p^t)}$ also commutes with  $Y^{(p^s)}X^{(p^s)}$. 

(ii) It is easy to see that $\mathcal{A}$ consists of all the elements of degree $0$ in $\mathcal{U}$. Let $C_{\mathcal{U}}(\mathcal{U}^0)$ be the centralizer 
of $\mathcal{U}^0$ in $\mathcal{U}$. Clearly 
$\mathcal{A} \subseteq C_{\mathcal{U}}(\mathcal{U}^0)$, and we have to show that 
$\mathcal{A} \supseteq C_{\mathcal{U}}(\mathcal{U}^0)$. Let $z$ be an element of 
$C_{\mathcal{U}}(\mathcal{U}^0)$. Then $z$ can be written uniquely as 
$$z = \sum_{m, m' \in \mathbb{Z}_{\geq 0}} Y^{(m)} z_{(m, m')} X^{(m')}$$
with all $ z_{(m, m')} \in \mathcal{U}^0$, almost all equal to $0$. To show that 
$z \in \mathcal{A}$, it is enough to show that $z_{(m,m')} =0$ for each pair 
$(m,m')$ with $m \neq m'$. For $s \in \mathbb{Z}$ and $t \in \mathbb{Z}_{\geq 0}$, 
we have 
$${H+s \choose t} z = 
\sum_{m, m' \in \mathbb{Z}_{\geq 0}} Y^{(m)} {H+s-2m \choose t} z_{(m, m')} X^{(m')},$$
$$z {H+s \choose t} =  
\sum_{m, m' \in \mathbb{Z}_{\geq 0}} Y^{(m)} {H+s-2m' \choose t} z_{(m, m')} X^{(m')}.$$
Since ${H+s \choose t} z= z {H+s \choose t}$, for each pair $(m,m')$, we must have 
$$\bigg(  {H+s-2m \choose t} - {H+s-2m' \choose t} \bigg)  z_{(m, m')} =0$$
for any $s \in \mathbb{Z}$ and $t \in \mathbb{Z}_{\geq 0}$. Take unique integers $c \in \mathbb{Z}_{\geq 0}$ 
and $d \in \mathbb{Z}$ such that 
$2(m'-m)= p^c d$ and $p$ does not divide $d$. Then if we take $s= 2m'$ and $t=p^c$, we have 
\begin{eqnarray*}
 {H+s-2m \choose t} - {H+s-2m' \choose t}
& = &  {H+p^c d \choose p^c} - {H \choose p^c} \\
& = & \sum_{i=0}^{p^c} {p^c d \choose i} {H \choose p^c-i} - {H \choose p^c} \\
& = &  {H \choose p^c}+d- {H \choose p^c} =d. 
\end{eqnarray*}  
Since $d \neq 0 $ in $\mathbb{F}_p$, we must have $z_{(m,m')}=0$ and hence 
$z \in \mathcal{A}$. 

(iii)  It is easy from the fact that $X^{(m)} {H \choose n} Y^{(m)}$ with   
$m,n \in \mathbb{Z}_{\geq 0}$ as well as $Y^{(m)} {H \choose n} X^{(m)}$ with   
$m,n \in \mathbb{Z}_{\geq 0}$ form a $k$-basis of $\mathcal{A}$. 
\end{proof} 
\

For $r,r' \in \mathbb{Z}_{>0}$ with $r'>r$, we write $\mathcal{A}_r$ 
(resp. $\widetilde{\mathcal{A}}_r$, $\mathcal{A}_{r,r'}$) for the 
subalgebra of $\mathcal{U}_r$ (resp.  $\widetilde{\mathcal{U}}_r$, $\mathcal{U}_{r,r'}$) generated by $\mathcal{U}_r^0$ (resp. $\mathcal{U}^0$, $\mathcal{U}_{r'}^0$) and the elements 
$Y^{(p^i)}X^{(p^i)}$ with $0 \leq i \leq r-1$. Clearly we have  
$\mathcal{A}_r = \mathcal{A} \cap \mathcal{U}_r$, 
$\widetilde{\mathcal{A}}_r = \mathcal{A} \cap \widetilde{\mathcal{U}}_r$ and 
$\mathcal{A}_{r,r'} = \mathcal{A} \cap \mathcal{U}_{r,r'}$. We also see that 
$\mathcal{A}_r$ (resp. $\widetilde{\mathcal{A}}_r$, $\mathcal{A}_{r,r'}$) is 
the centralizer of $\mathcal{U}_r^0$ (resp. $\mathcal{U}^0$, $\mathcal{U}_{r'}^0$) in $\mathcal{U}_r$ (resp.  $\widetilde{\mathcal{U}}_r$, $\mathcal{U}_{r,r'}$) and is free over $\mathcal{U}_r^0$  (resp. $\mathcal{U}^0$, $\mathcal{U}_{r'}^0$) of basis $X^{(m)} Y^{(m)}$, 
$0 \leq m \leq p^r-1$, and also of basis $Y^{(m)}X^{(m)}$, 
$0 \leq m \leq p^r-1$. 

The following proposition will be  used  in Section 5. 

\begin{Proposition}
Let $n \in \mathbb{Z}_{\geq 0}$. Then $X^{(np)}$ and 
$Y^{(np)}$ commute with all elements in $\mathcal{A}_1$. 
\end{Proposition}

\begin{proof}
We may assume that $n>0$. Then  
\begin{eqnarray*}
X^{(np)} Y X
& = & YX^{(np)} X + (H-np+1) X^{(np-1)} X \\
& = & Y X X^{(np)}
\end{eqnarray*} 
and 
$$X^{(np)} H= (H-2np)  X^{(np)}  = H X^{(np)} .$$
Similarly we have $Y X Y^{(np)}= Y^{(np)} Y X$ and 
$Y^{(np)} H= H Y^{(np)} $, as desired. 
\end{proof}

\section{Representation theory of $\mathcal{U}$}
We describe some elementary facts on the representation theory of $\mathcal{U}$ and 
some subalgebras.  

In this paper, all modules we consider are assumed to be finite-dimensional left modules. The category of finite-dimensional $\mathcal{U}$-
(resp.  $\mathcal{U}_r$-) modules  
is identified with that of finite-dimensional (rational) $G$-
(resp. $G_r$-) modules. 

For a nonzero $\mathcal{U}^0$-module $V$ and an integer $\lambda \in \mathbb{Z}$, set
$$V_{\lambda}= \bigg\{ v \in V \ \bigg|\ {H \choose m} v = {\lambda \choose m} v,\ 
\forall m \in \mathbb{Z}_{\geq 0} \bigg\}.$$
This is a $\mathcal{U}^0$-submodule of $V$. If $V_{\lambda} \neq 0$, then 
we call it the $\mathcal{U}^0$-weight space of ($\mathcal{U}^0$-) weight $\lambda$. 
A nonzero element of $V_{\lambda}$ is called a $\mathcal{U}^0$-weight vector of  
weight $\lambda$. Any $\mathcal{U}$-module $M$ can be written as a direct sum of 
its $\mathcal{U}^0$-weight spaces: $M= \bigoplus_{\lambda \in \mathbb{Z}} M_{\lambda}$.  
For $\lambda \in \mathbb{Z}$, let $k_{\lambda}$ be a one-dimensional 
$\mathcal{U}^0$-module which is also a $\mathcal{U}^0$-weight space of weight $\lambda$. 

As in the case of $\mathcal{U}^0$-modules, we can define a notion of weights for 
$\mathcal{U}_r^0$-modules. For an integer $\mu \in \mathbb{Z}$, 
we consider the subspace $V_{(r, \mu)}$ of a nonzero $\mathcal{U}_r^0$-module 
$V$ consisting of the vectors $v \in V$ which satisfy 
${H \choose m} v = {\mu \choose m} v$ for all integers $m$ with $0 \leq m \leq p^r-1$. 
If  $V_{(r, \mu)} \neq 0$, then we  call it the $\mathcal{U}_r^0$-weight space of 
$\mathcal{U}_r^0$-weight $\mu$. If $\nu \equiv \mu \ ({\rm mod}\ p^r)$, then we have 
$V_{(r, \nu)}=V_{(r, \mu)}$ since 
${\nu \choose m} \equiv {\mu \choose m}\ ({\rm mod}\ p)$ for all $m$ 
with $0 \leq m \leq p^r-1$. Therefore, a $\mathcal{U}_r^0$-weight 
$\mu \in \mathbb{Z}$ can be  regarded as an element of $\mathbb{Z}/ p^r \mathbb{Z}$ and a $\mathcal{U}_r^0$-module $V$ is decomposed as 
$V= \bigoplus_{\mu =0}^{p^r-1} V_{(r, \mu)} = 
\bigoplus_{\mu  \in \mathbb{Z}/ p^r \mathbb{Z}} V_{(r, \mu)}$. 
Moreover,  
if $M$ is a $\mathcal{U}$-module, then  the $\mathcal{U}_r^0$-weight space 
$M_{(r, \mu)}$ of $\mathcal{U}_r^0$-weight $\mu \in \mathbb{Z}$ is decomposed as 
$M_{(r, \mu )}= \bigoplus_{\lambda} M_{\lambda}$ where $\lambda$ runs through the integers 
with $\lambda \equiv \mu \ ({\rm mod}\ p^r)$.

For a $\mathcal{U}$-module $M$ and $i \in \mathbb{Z}_{\geq 0}$, 
let  $M^{[i]}$ be another $\mathcal{U}$-module defined as follows: it is equal to $M$ as a 
$k$-vector space and the action of $z \in \mathcal{U}$ on $M^{[i]}$ is induced by that of 
${\rm Fr}^i(z)$ on $M$. Therefore, if $v^{[i]}$ is  the corresponding element in 
$M^{[i]}$ for $v \in M$, then  
$z v^{[i]} = \big({\rm Fr}^i(z) v\big)^{[i]}$ for $z \in \mathcal{U}$. The $\mathcal{U}$-module 
$M^{[i]}$ is called the $i$-th Frobenius twist of $M$. If an element $v \in M$ has 
$\mathcal{U}^0$-weight $\lambda$, the corresponding element $v^{[i]} \in M^{[i]}$ has $\mathcal{U}^0$-weight 
$ \lambda p^i$.  Moreover, if $i \geq r$,  $M^{[i]}$ is isomorphic to a direct sum of 
${\rm dim}_k M$ copies of the 
trivial module $k$ as a $\mathcal{U}_r$-module. 
If $M$ and $M'$ are $\mathcal{U}$-modules, the tensor product $M \otimes M'$ 
is again a $\mathcal{U}$-module via the comultiplication on $\mathcal{U}$
$$X^{(m)}(v \otimes v') = \sum_{i=0}^m X^{(i)} v \otimes X^{(m-i)} v'$$ 
and 
$$Y^{(m)}(v \otimes v') = \sum_{i=0}^m Y^{(i)} v \otimes Y^{(m-i)} v'$$ 
for $m \in \mathbb{Z}_{\geq 0}$, $v \in M$ and $v' \in M'$. 

Let $V_{\mathbb{C}}(\lambda)$ be a simple $\mathcal{U}_{\mathbb{C}}$-module with 
highest weight $\lambda \in \mathbb{Z}_{\geq 0}$. For a fixed highest weight vector 
$v_{\lambda} \in V_{\mathbb{C}}(\lambda)$, the $\mathcal{U}$-module 
$V(\lambda)=k \otimes_{\mathbb{Z}} (\mathcal{U}_{\mathbb{Z}} v_{\lambda})$ 
is called the 
Weyl ($\mathcal{U}$-) module with highest weight $\lambda$. For an element of 
$\mathcal{U}_{\mathbb{Z}} v_{\lambda}$, we use the same notation for its image in $V(\lambda)$. 
The vectors $Y^{(i)} v_{\lambda}$ with $0 \leq i \leq \lambda$ have  weight $\lambda-2i$ 
and form a basis of $V(\lambda)$. Each Weyl module $V(\lambda)$ has a unique 
maximal submodule, and the  quotient $L(\lambda)$ of $V(\lambda)$ by the submodule 
is a simple $\mathcal{U}$-module. Then all 
$L(\lambda)$ with $\lambda \in \mathbb{Z}_{\geq 0}$ form the set of non-isomorphic simple $\mathcal{U}$-modules. 
For $\lambda \in \mathbb{Z}_{\geq 0}$ and its $p$-adic expansion 
$\lambda= \sum_{i=0}^{n-1} \lambda_i p^i$, we have by Steinberg's tensor product 
theorem that 
$$L(\lambda) \cong L(\lambda_0) \otimes L(\lambda_1)^{[1]} \otimes \cdots 
\otimes L(\lambda_{n-1})^{[n-1]}$$
as $\mathcal{U}$-modules. 
For $i \in \mathbb{Z}_{\geq 0}$, the simple 
module $L(p^i-1)$ is called the $i$-th Steinberg module and often denoted by 
${\rm St}_i$. The trivial $\mathcal{U}$-module $k$ is isomorphic to $L(0)$, and 
the $r$-th Steinberg module ${\rm St}_r$ is projective as a $\mathcal{U}_r$-module. If 
$0 \leq \lambda \leq p^r-1$, then $L(\lambda)$ is also simple as a 
$\mathcal{U}_r$-module, and any simple $\mathcal{U}_r$-module can be obtained 
in this way. For an integer $\lambda$ with $0 \leq \lambda \leq p^r-1$, let $Q_r(\lambda)$ be the projective cover of the simple $\mathcal{U}_r$-module $L(\lambda)$. 
It is known that these $\mathcal{U}_r$-modules $Q_r(\lambda)$ can be extended to  $\mathcal{U}$-modules for any $r \in \mathbb{Z}_{> 0}$ 
in this situation $G={\rm SL}_2(k)$ (see \cite[4.5 Corollar]{jantzen80} or 
\cite[II. 11.11]{jantzenbook}). If $0 \leq \lambda \leq p-2$, then $Q_1(\lambda)$ 
is a uniserial $\mathcal{U}$-module with 
$$Q_1(\lambda)/ {\rm rad}_{\mathcal{U}}Q_1(\lambda) \cong L(\lambda),$$ 
$${\rm rad}_{\mathcal{U}}Q_1(\lambda)/{\rm soc}_{\mathcal{U}}Q_1(\lambda) 
\cong L(2p-2-\lambda)$$
and 
$${\rm soc}_{\mathcal{U}}Q_1(\lambda) \cong L(\lambda),$$
whereas $Q_1(p-1)= L(p-1)={\rm St}_1$. As in the simple 
$\mathcal{U}$-modules, for $\lambda \in \mathbb{Z}_{\geq 0}$ and its $p$-adic expansion 
$\lambda= \sum_{i=0}^{n-1} \lambda_i p^i$, we have 
$$Q_n(\lambda) \cong Q_1(\lambda_0) \otimes Q_1(\lambda_1)^{[1]} \otimes \cdots 
\otimes Q_1(\lambda_{n-1})^{[n-1]}$$
as $\mathcal{U}$-modules. 
The highest $\mathcal{U}^0$-weight in $Q_n(\lambda)$ is $2p^n-2-\lambda$ and 
the lowest one is $-2p^n+2+\lambda$ (see \cite[(2.2) Example 1]{donkin93}).

For later use we also consider some modules for $\widetilde{\mathcal{U}}_r$ and 
$\mathcal{U}_{r,r'}$ for an integer $r'$ greater than $r$. Let 
$\lambda',\lambda'' \in \mathbb{Z}$ be integers with $0 \leq \lambda' \leq p^r-1$ and set 
$\lambda = \lambda' +  \lambda''p^r$. The simple $\mathcal{U}$-module 
$L(\lambda')$ is also 
simple as a $\widetilde{\mathcal{U}}_r$-module, and then 
$\widetilde{L}_r(\lambda) = L(\lambda') \otimes k_{ \lambda''p^r}$ is a simple 
$\widetilde{\mathcal{U}}_r$-module with highest $\mathcal{U}^0$-weight 
$\lambda$, where the one-dimensional $\mathcal{U}^0$-module $k_{ \lambda''p^r}$ 
is regarded as a $\widetilde{\mathcal{U}}_r$-module by the trivial action of $\mathcal{U}_r$.  
Then we see that  $\widetilde{L}_r(\lambda) \cong L(\lambda')$ as $\mathcal{U}_r$-modules. Similarly, set 
$\widetilde{Q}_r(\lambda) = Q_r(\lambda') \otimes k_{ \lambda''p^r}$ for 
the above $\lambda=\lambda' +  \lambda''p^r$. 
This is the projective cover of the simple $\widetilde{\mathcal{U}}_r$-module 
$\widetilde{L}_r(\lambda)$, and we have $\widetilde{Q}_r(\lambda) \cong Q_r(\lambda')$ as 
$\mathcal{U}_r$-modules. The simple $\widetilde{\mathcal{U}}_r$-modules $\widetilde{L}_r(\nu)$ with $\nu \in \mathbb{Z}$ are simple as $\mathcal{U}_{r,r'}$-modules, 
and all simple $\mathcal{U}_{r,r'}$-modules can be obtained in this way. 
Since $\widetilde{L}_r(\nu_1) \cong \widetilde{L}_r(\nu_2)$ as 
$\mathcal{U}_{r,r'}$-modules if and only if 
$\nu_1 \equiv \nu_2\ ({\rm mod}\ p^{r'})$, all non-isomorphic simple $\mathcal{U}_{r,r'}$-modules can be written as $\widetilde{L}_r(\nu)$,  
$0 \leq \nu \leq p^{r'}-1$.  Then the simple $\mathcal{U}_{r,r'}$-module $\widetilde{L}_r(\nu)$ has $\widetilde{Q}_r(\nu)$ as its projective cover.

The following proposition will be used later to determine the PIMs for 
$\mathcal{U}_{r,r'}$ (hence for $\mathcal{U}_r$) 
which the idempotents given there generate. 

\begin{Proposition} 
Let $r'$ be an integer which is greater than $r$.  
Let $\lambda$ be an integer with $0 \leq \lambda \leq p^{r'}-1$, and let $v$ be a 
$\mathcal{U}_{r'}^0$-weight vector of $\mathcal{U}_{r'}^0$-weight 
$\nu$ with $0 \leq \nu \leq p^{r'}-1$ which is also a generator of  the projective indecomposable $\mathcal{U}_{r,r'}$-module 
$\widetilde{Q}_r(\lambda)$. Let 
$t$ be the largest integer $n$ with $X^{(n)} v \neq 0$ and $0 \leq n \leq p^r-1$. Then if we write $\lambda= \lambda'+ \lambda''p^r$ and $\nu=\nu' + \nu''p^r$ for some unique 
integers $\lambda', \lambda'',\nu',\nu''$ with $0 \leq \lambda',\nu' \leq p^r-1$ and 
$0 \leq \lambda'',\nu'' \leq p^{r'-r}-1$, the following holds: \\ \\
{\rm (i)} $t= p^r-1-(\lambda'+\nu')/2$ and $\lambda''=\nu''$ if $\nu'+2t \leq 2p^r-2$, \\
{\rm (ii)} $t= 3p^r/2-1-(\lambda'+\nu')/2$ and $\lambda''=\nu''+1$ 
if $\nu'+2t > 2p^r-2$ and $\nu'' \neq p^{r'-r}-1$, \\
{\rm (iii)} $t= 3p^r/2-1-(\lambda'+\nu')/2$ and $\lambda''=0$ if $\nu'+2t > 2p^r-2$ 
and $\nu'' = p^{r'-r}-1$. 
\end{Proposition}
\begin{proof}
Recall that the highest $\mathcal{U}^0$-weight in $\widetilde{Q}_r(\lambda)$ is 
$2p^r-2-\lambda'+ \lambda''p^r$ and the lowest one is 
$-2p^r+2+\lambda'+ \lambda''p^r$ if we extend the $\mathcal{U}_{r,r'}$-module 
$\widetilde{Q}_r(\lambda)$  
to a $\widetilde{\mathcal{U}}_r$-module. 
Hence $v$ has the $\mathcal{U}^0$-weight space decomposition 
$$v=v_{-1} + v_{0} + v_{1},$$
where  $v_i \in \widetilde{Q}_r(\lambda)_{\nu +i p^{r'}}$.  Since $v$ generates 
$\widetilde{Q}_r(\lambda)$, 
$v$ does not lie in 
${\rm rad}_{\mathcal{U}_{r,r'}} \widetilde{Q}_r(\lambda)={\rm rad}_{\widetilde{\mathcal{U}}_r} \widetilde{Q}_r(\lambda)$. 
Note that  
$\widetilde{Q}_r(\lambda)/ {\rm rad}_{\widetilde{\mathcal{U}}_r}\widetilde{Q}_r(\lambda) 
\cong \widetilde{L}_r(\lambda)$ and that any 
$\mathcal{U}^0$-weight $\gamma$ of $\widetilde{L}_r(\lambda)$ satisfies 
$-\lambda'+ \lambda''p^r \leq \gamma \leq \lambda'+ \lambda''p^r$. 
Since $\nu +p^{r'}$ is not a $\mathcal{U}^0$-weight of $\widetilde{L}_r(\lambda)$, we have 
$v_1 \in {\rm rad}_{\widetilde{\mathcal{U}}_r} \widetilde{Q}_r(\lambda)$. 
Therefore, we see that $v_{-1}$ or 
$v_{0}$ does not lie in ${\rm rad}_{\widetilde{\mathcal{U}}_r} \widetilde{Q}_r(\lambda)$.   

Suppose that $v_{-1} \not\in {\rm rad}_{\widetilde{\mathcal{U}}_r}\widetilde{Q}_r(\lambda)$. Then 
since $-\lambda'+ \lambda''p^r \leq \nu-p^{r'} \leq \lambda'+ \lambda''p^r$, 
we must have $\lambda''=0$ and $\nu-p^{r'} \geq -\lambda'$. 
 Therefore, since $$\nu \geq p^{r'}-\lambda' > 2(p^r-1) -\lambda' 
=2(p^r-1) -\lambda'+ \lambda''p^r,$$
we have $v_0=v_{1}=0$ and $v=v_{-1}$. Note  that the inequality 
$\nu \geq p^{r'}-\lambda'$ also implies 
$\nu''=p^{r'-r}-1$. 
Since $v \not\in {\rm rad}_{\widetilde{\mathcal{U}}_r} \widetilde{Q}_r(\lambda)$, 
there exists 
an element $z \in \widetilde{\mathcal{U}}_r$ such that $z v$ is a highest $\mathcal{U}^0$-weight vector in  
$\widetilde{Q}_r(\lambda)$ (with   $\mathcal{U}^0$-weight $2p^r-2- \lambda'+\lambda''p^{r}$).  
Moreover, since $v$ has 
$\mathcal{U}^0$-weight $\nu-p^{r'}$, the element $z$ can be taken as a linear combination 
of some elements of the form 
$$Y^{(m_1)} h X^{(m_2)}$$
with $h \in \mathcal{U}^0$, $0 \leq m_1, m_2 \leq p^r-1$ and  
$$m_2-m_1 = (2p^r-2-\lambda'+\lambda''p^r -(\nu-p^{r'}))/2 = 
3p^r/2-1-(\lambda'+\nu')/2.$$
This implies that $z v$ is proportional to 
$X^{( 3p^r/2-1-(\lambda'+\nu')/2)} v $ and that  
$t=3p^r/2-1-(\lambda'+\nu')/2$. In this case we also see that 
$\nu'+2t = 3p^r-2-\lambda' > 2p^r-2$. 

Finally suppose that $v_{0} \not\in {\rm rad}_{\mathcal{U}}\widetilde{Q}_r(\lambda)$. 
Then $-\lambda'+ \lambda''p^r \leq \nu \leq \lambda'+ \lambda''p^r$. Since 
$\nu-p^{r'} < -2(p^r-1)+\lambda'+ \lambda''p^r$ and 
$\nu+p^{r'} > 2(p^r-1)-\lambda'+ \lambda''p^r$, we must have 
$v_1=v_{-1}=0$ and $v=v_0$. Moreover, $\lambda''$ must be equal to $\nu''$ or 
$\nu''+1$ since 
$-\lambda' \leq \nu' +(\nu'' -\lambda'')p^r \leq \lambda'$.   
As in the last paragraph, we see that  $\lambda''=\nu''$ if and only if  
$$t=(2(p^r-1) -\lambda' + \lambda''p^r-\nu)/2 = p^r-1 -(\lambda'+\nu')/2,$$
and in this case we have $\nu'+2t = 2p^r-2-\lambda' \leq 2p^r-2$. 
Similarly, we see that $\lambda''=\nu''+1$ if and only if  
$$t=(2(p^r-1) -\lambda' + \lambda''p^r-\nu)/2 = 3p^r/2-1 -(\lambda'+\nu')/2,$$
and in this case we have $\nu'' \neq p^{r'-r}-1$ and $\nu'+2t = 3p^r-2-\lambda' > 2p^r-2$. 
Therefore,  the proposition is proved.  
\end{proof}

\section{Primitive idempotents in $\mathcal{U}_1$}
We shall construct primitive idempotents in $\mathcal{U}_1$, 
following Seligman's method \cite{seligman03}. 

For $a \in \mathbb{Z}$, set 
$$\mu_a = {H-a-1 \choose p-1} = \sum_{i=0}^{p-1}
{-a-1 \choose p-1-i} {H \choose i} \in \mathcal{U}_1^0.$$
Then $\mu_a = \mu_b$ if and only if $a \equiv b\ ({\rm mod}\ p)$. Therefore, 
the integer $a$ in the symbol $\mu_a$ can be regarded as an element of 
$\mathbb{F}_p = \mathbb{Z}/p \mathbb{Z}$. Then we obtain by Wilson's theorem 
that 
\begin{eqnarray*}
\mu_a & = & \dfrac{(H-a-1)(H-a-2) \cdots (H-a-p+1)}{(p-1)!} \\
         & = & -(H-a-1)(H-a-2) \cdots (H-a-p+1) \\
         & = & - \prod_{\gamma \in \mathbb{F}_p -\{ a \}} (H - \gamma).
\end{eqnarray*}
It is easy to check the following facts (see  \cite[\S 4]{gros-kaneda15}). 

\begin{Proposition}
{\rm (i)} For $a \in \mathbb{Z}$, we have $H \mu_a = a \mu_a$. Therefore, $\mu_a$ is 
a $\mathcal{U}_1^0$-weight vector of weight $a$ in the $\mathcal{U}_1^0$-module  
$\mathcal{U}_1^0$. \\ \\
{\rm (ii)} The elements $\mu_a$ with $a \in \mathbb{F}_p$ are pairwise orthogonal 
primitive idempotents in $\mathcal{U}_1^0$ whose sum is $1 \in \mathcal{U}_1^0$. \\ \\
{\rm (iii)} $\mu_{a} X^{(m)} = X^{(m)} \mu_{a-2m}$ and $\mu_{a} Y^{(m)} = Y^{(m)} \mu_{a+2m}$ 
for $a \in \mathbb{Z}$ and $m \in \mathbb{Z}_{\geq 0}$. 
\end{Proposition}

To construct primitive idempotents in $\mathcal{U}_1$, we need some lemmas. 

\begin{Lemma}
For $m \in \mathbb{Z}_{> 0}$ and $a \in \mathbb{F}_p$, we have 
$$\mu_a Y^m X^m = \prod_{i=0}^{m-1} \big( \mu_a YX - i(i+a+1) \big)$$
and
$$\mu_a X^m Y^m = \prod_{i=0}^{m-1} \big( \mu_a XY - i(i-a+1) \big)$$
in $\mathcal{U}_1$.
\end{Lemma}

\begin{proof}
It is clear when $m=1$, and we may assume $m \geq 2$. Then we have 
\begin{eqnarray*}
\lefteqn{\mu_a Y^m X^m} \\
& = & \mu_a Y^{m-1} Y X^{m-1} X \\
& = & \mu_a Y^{m-1} \big( X^{m-1} Y+(m-1)X^{m-2} (-H-m+2) \big) X \\
& = & \mu_a Y^{m-1} X^{m-1} YX + (m-1) \mu_a Y^{m-1} X^{m-1}(-H-m) \\
& = & \mu_a Y^{m-1} X^{m-1} YX + (m-1) \mu_a (-H-m) Y^{m-1} X^{m-1} \\
& = &  \mu_a Y^{m-1} X^{m-1} YX + (m-1) (-a-m) \mu_a Y^{m-1} X^{m-1} \\
& = & \mu_a Y^{m-1} X^{m-1} \big( \mu_a YX -(m-1)(a+m) \big).
\end{eqnarray*}
By induction on $m$, this is equal to 
$$\prod_{i=0}^{m-2}\big(\mu_a YX -i(i+a+1)\big) \big(\mu_a YX -(m-1)(a+m)\big)
=\prod_{i=0}^{m-1}\big(\mu_a YX -i(i+a+1)\big).$$
Similarly we can check that $\mu_a X^m Y^m= \prod_{i=0}^{m-1} \big(\mu_a XY - i(i-a+1)\big)$.
\end{proof}

If $p$ is odd, set $\mathcal{S} = \{ i \in \mathbb{F}_p \ | \ i =0, 1, \cdots , (p-1)/2 \}$.  
Then we define  polynomials $\varphi_{a,m}(x), \psi(x) \in \mathbb{F}_p [x]$ 
for $a \in \mathbb{F}_p$ and $m \in \mathbb{Z}_{\geq 0}$ as 
$$\varphi_{a,0}(x)= 1,$$
$$\varphi_{a,n}(x)= \prod_{i=0}^{n-1} \big(x-i(i+a+1)\big)$$
if $n >0$, and 
$$\psi(x)=\prod_{i \in \mathbb{F}_p} (x- i^2) = x \prod_{i \in \mathcal{S}- \{ 0\}} 
(x- i^2)^2.$$

\begin{Lemma} Suppose that $p$ is odd and let $a \in \mathbb{F}_p$.  Then the following holds. \\ \\
{\rm (i)} $\psi \Big(x+\big((a+1)/2\big)^2\Big) = \varphi_{a,p}(x)$. \\
{\rm (ii)} $\varphi_{a,p} (\mu_a YX) = \varphi_{-a, p}(\mu_a X Y) =0$. 
\end{Lemma}

\begin{proof}
(i) Set $y = x + \big((a+1)/2\big)^2$. We have 
\begin{eqnarray*}
\varphi_{a,p}(x) & = & \prod_{i \in \mathbb{F}_p} \big(x-i(i+a+1)\big) \\
                    & = & \prod_{i \in \mathbb{F}_p} \Big(x+\big((a+1)/2\big)^2- 
                             \big(i+(a+1)/2\big)^2\Big) \\
                     & = &  \prod_{i \in \mathbb{F}_p} \Big(y- \big(i+(a+1)/2\big)^2\Big) \\
                     & = &  \prod_{i \in \mathbb{F}_p} (y- i^2) \\
                     & = &  \psi(y).
\end{eqnarray*}
(ii) It is immediate from Lemma 4.2 since $X^p=Y^p=0$ in $\mathcal{U}$. 
\end{proof}
If $p$ is odd, we define $(p+1)/2$ polynomials $\psi_0(x)$ and $\psi_j(x)$ with 
$j \in \mathcal{S}- \{ 0\} $  as 
$$\psi_0(x)= \prod_{i \in \mathbb{F}_p^{\times}} (x-i^2) = 
\prod_{i \in \mathcal{S}-\{ 0\}} (x-i^2)^2$$
and 
$$\psi_{j}(x) = 2x(x+j^2) \prod_{i \in \mathbb{F}_p^{\times} - \{ j, p-j \} } (x-i^2)
 = 2x(x+j^2) \prod_{i \in \mathcal{S} - \{0, j \} } (x-i^2)^2.$$

\begin{Lemma} 
Suppose that $p$ is odd. 
The following holds in $\mathbb{F}_p [x]$ ($\delta_{m,n}$ denotes Kronecker's symbol). \\ \\
{\rm (i)} $\sum_{i \in \mathcal{S}} \psi_{i}(x)=1 $. \\
{\rm (ii)} $\psi_m(x) \psi_{n}(x) \equiv \delta_{m,n} \psi_m (x)\ \Big({\rm mod}\ 
\big(\psi(x)\big) \Big)$ for 
 $m,n \in \mathcal{S}$. 
\end{Lemma}

\begin{proof}
(i) Set $\Phi(x)= \sum_{i \in \mathcal{S} } \psi_i(x) -1$. Then we have 
$$\dfrac{d \psi_0}{d x} = \sum_{t \in \mathcal{S}- \{ 0\}} 2(x-t^2) 
\prod_{i \in \mathcal{S} - \{0, t \} } 
(x-i^2)^2$$
and
$$\dfrac{d \psi_j}{d x} = 2(2x+j^2) \prod_{i \in \mathcal{S} - \{0, j \} } (x-i^2)^2
+2x(x+j^2) \sum_{t \in \mathcal{S} - \{0, j\} } 
2(x-t^2) \prod_{i \in \mathcal{S} - \{0, j,t\} }
(x-i^2)^2$$
for $j \in \mathcal{S} -\{ 0\}$. Suppose that $\Phi(x) \neq 0$ in 
$\mathbb{F}_p [x]$. 
If $s \in \mathcal{S}-\{ 0\}$ we have 
\begin{eqnarray*}
\Phi(s^2) 
& = & \sum_{i \in \mathcal{S}  } \psi_i(s^2) -1 \\
& = & \psi_s(s^2) -1 \\
& = & 2s^2 (s^2+s^2) \prod_{i \in \mathbb{F}_p^{\times} - \{ s, p-s\} } (s^2-i^2)-1 \\
& = & 4s^4\prod_{ i \in \mathbb{F}_p^{\times} - \{ s, p-s\} } (s+i)(s-i)-1 \\
& = & 4s^4 \bigg( \prod_{ i \in \mathbb{F}_p^{\times} - \{ s, p-s\} } (s+i) \bigg)
\bigg( \prod_{j \in \mathbb{F}_p^{\times} - \{ s, p-s\} } (s-j) \bigg) -1 \\
& = & 4s^4 \bigg( \dfrac{1}{2s^2} \prod_{i \in \mathbb{F}_p -\{  p-s\} } (s+i) \bigg)
\bigg( \dfrac{1}{2s^2} \prod_{j \in \mathbb{F}_p - \{ s\} } (s-j) \bigg) -1 \\
& = & 4s^4 \bigg( \dfrac{1}{2s^2} \prod_{i \in \mathbb{F}_p^{\times}} i \bigg)
\bigg( \dfrac{1}{2s^2} \prod_{j \in \mathbb{F}_p^{\times}} j \bigg) -1 \\
& = & 4s^4 \cdot \bigg( - \dfrac{1}{2s^2} \bigg)^2 -1 \\
& = & 0,
\end{eqnarray*}
whereas 
\begin{eqnarray*}
\Phi(0) & = & \sum_{i \in \mathcal{S} } \psi_i(0) -1 \\
& = & \psi_0(0) -1 \\
& = & \prod_{i \in \mathbb{F}_p^{\times}} (-i^2)-1 \\
& = & 1-1 \\
& = & 0.
\end{eqnarray*}
Therefore, $\Phi(x)$ contains each linear polynomial $x-s^2$ with $s \in \mathcal{S} $ as a factor. 

We would like to claim that 
$ \left. \dfrac{d \psi_j}{d x} \right|_{x=s^2} =0$ 
if $j \in \mathcal{S}$ and $s \in \mathcal{S} - \{ 0\}$. Clearly the equality holds 
if $j = 0$ or if $j \neq 0$ and $s \neq j$. So we only have to show that 
$\left. \dfrac{d \psi_j}{d x} \right|_{x=j^2} =0$ for $j \in \mathcal{S}- \{ 0\}$. 
If $j \in \mathcal{S}-\{0\}$, we have 
\begin{eqnarray*}
\left. \dfrac{d \psi_j}{d x} \right|_{x=j^2} 
& = & 6j^2 \prod_{i \in \mathcal{S} - \{0, j \} } (j^2 -i^2)^2 + 
 4j^4 \sum_{t \in \mathcal{S} - \{0, j\} } 2(j^2-t^2) 
\prod_{i \in \mathcal{S} - \{0, j,t \} } (j^2-i^2)^2 \\
& = & \bigg( 6j^2 +4j^4 \sum_{t \in \mathcal{S} - \{0, j \} } \dfrac{2}{j^2-t^2} \bigg) 
\prod_{i \in \mathcal{S} - \{0, j\} } (j^2 -i^2)^2.
\end{eqnarray*} 
But since
\begin{eqnarray*}
\lefteqn{\sum_{t \in \mathcal{S} -\{0, j\}} \dfrac{2}{j^2-t^2}} \\
& = & \dfrac{1}{2} \sum_{t \in \mathbb{F}_p^{\times} - \{ j, p-j \} } \dfrac{2}{j^2-t^2} \\
& = & \dfrac{1}{2j} \sum_{t \in \mathbb{F}_p^{\times} - \{ j, p-j \} } 
\bigg( \dfrac{1}{j-t} +\dfrac{1}{j+t} \bigg) \\
& = & \dfrac{1}{2j} \Bigg( \sum_{t \in \mathbb{F}_p - \{ j\} }
 \dfrac{1}{j-t} -\dfrac{1}{j}-\dfrac{1}{2j}  +  
\sum_{t \in \mathbb{F}_p - \{ p-j\} } 
 \dfrac{1}{j+t} -\dfrac{1}{j}-\dfrac{1}{2j}  \Bigg) \\
& = & \dfrac{1}{2j} \Bigg(  0-\dfrac{1}{j}-\dfrac{1}{2j}  + 
 0- \dfrac{1}{j} - \dfrac{1}{2j}  \Bigg) \\
& = & -\dfrac{3}{2j^2}, 
\end{eqnarray*} 
we conclude that $\left. \dfrac{d \psi_j}{d x} \right|_{x=j^2} =0$, and the claim 
follows. Therefore, 
we obtain   $\left. \dfrac{d \Phi}{d x} \right|_{x=s^2} =0$ for 
$s \in \mathcal{S}-\{0\}$ and 
the polynomial $\Phi(x)$ has $x \prod_{i \in \mathcal{S}-\{0\} } (x-i^2)^2$ as a factor. So  the degree of  the polynomial $\Phi(x)$ is greater than $p-1$, 
which is a contradiction. Therefore, $\Phi(x)$ must be zero and  
(i) is proved. 

(ii) By the definition of $\psi_i(x)$, clearly we have $\psi_m(x) \psi_n(x) \in \big(\psi(x)\big)$ 
for $m \neq n$. Combining this with (i) the result follows.   
\end{proof}

Set
$\mathcal{P}= \mathbb{F}_p \times \mathcal{S}$ if $p$ is odd, and 
$\mathcal{P} = \{(0,1/2), (1,0), (1,1) \} \subset \mathbb{F}_2 \times (1/2)\mathbb{Z}$ if $p=2$.

If $p$ is odd, for a pair $(a,j) \in \mathcal{P}$ we set 
$$E(a, j) = \psi_j \Big(\mu_a YX +\big((a+1)/2\big)^2\Big) \cdot \mu_a.$$
This element also can be written as 
$E(a, j) = \psi_j \Big(\mu_a XY +\big((a-1)/2\big)^2\Big) \cdot \mu_a$. 
If $p=2$, for each pair $(a,j) \in \mathcal{P}$ we define $E(a,j)$ as follows:
$$E(0,1/2)=\mu_0,\ \ E(1,0) = \mu_1YX=\mu_1 (XY+1),\ \ 
E(1,1) = \mu_1 XY = \mu_1 (YX+1).$$
Clearly  all $E(a,j)$ lie in $\mathcal{A}_1$. 

\begin{Proposition}
For a fixed element $a \in \mathbb{F}_p$, 
the elements $E(a,j)$ with  $(a,j) \in \mathcal{P} $ 
 are pairwise orthogonal primitive idempotents in $\mathcal{U}_1$ 
whose sum is $\mu_a$. Thus, all the elements $E(a,j)$ with 
$(a,j) \in \mathcal{P} $ are 
pairwise  orthogonal primitive idempotents in $\mathcal{U}_1$ 
whose sum is $1$. 
\end{Proposition}

\begin{proof}
If $p=2$, then the claim but primitivity is clear. Consider the case when $p$ is odd. 
Then note that 
$\mathcal{P}= \mathbb{F}_p \times \mathcal{S}$. For a fixed $a \in \mathbb{F}_p$, we have by Lemma 4.4 (i) that 
$$\sum_{j \in \mathcal{S}  } E(a,j) = 
\bigg( \sum_{j \in \mathcal{S}   } \psi_j
\Big(\mu_a YX + \big((a+1)/2\big)^2\Big) \bigg) \cdot \mu_a 
= \mu_a.$$
On the other hand, it follows from Lemma 4.4 (ii) and Lemma 4.3 (i) and (ii) that the elements $E(a,j)$ with  
$j \in \mathcal{S} $ are pairwise orthogonal idempotents in $\mathcal{U}_1$. 
Moreover, we also see that 
$$\sum_{a \in \mathbb{F}_p} \sum_{j \in \mathcal{S}  } E(a,j) =
\sum_{a \in \mathbb{F}_p} \mu_a = 1$$
and 
$$E(a,j) E(a', j') = \delta_{(a,j), (a', j')} E(a,j)$$
since $\mu_a \mu_{a'} = \delta_{a, a'} \mu_a$ by Proposition 4.1 (ii).  

It remains to show that each idempotent $E(a,j)$ is primitive when $p$ is arbitrary. 
Recall that the non-isomorphic simple $\mathcal{U}_1$-modules 
are $L(\lambda)$ with $0 \leq \lambda \leq p-1$ and that 
${\rm dim}_k L(\lambda)= \lambda+1$. Hence the number of summands in a decomposition of $1 \in \mathcal{U}_1$ into  
pairwise orthogonal primitive idempotents must be 
$$1+2 + \cdots +p = \dfrac{p(p+1)}{2}.$$
On the other hand, the number of all $E(a,j)$ with $(a,j) \in \mathcal{P}$ is also equal to $p(p+1)/2$. Therefore, each idempotent $E(a,j)$ 
must be primitive (see \cite[(54.5) Theorem]{curtis-reiner62}). 
\end{proof}

If $p$ is odd and $(a,j) \in \mathcal{P} $, 
we define $n(a,j)$ as the largest non-negative integer $n$ satisfying 
$\varphi_{a,n} (x) \mid \psi_j \Big(x + \big((a+1)/2\big)^2\Big)$. (Recall that 
$\varphi_{a,0} (x) =1$. Thus $n(a,j)=0$ if 
$x \nmid \psi_j \Big(x + \big((a+1)/2\big)^2\Big)$.) If $p=2$, set 
$$n(0,1/2)=0,\ \ n(1,0)=1,\ \ n(1,1)=0.$$

With respect to 
each pair  
$(a,j) \in \mathcal{P}$, we consider 
the following four cases, regarding $a$ and $j$ as the corresponding integers with 
$0 \leq a \leq p-1$ and $0 \leq j \leq (p-1)/2$ if $p$ is odd: \\ \\
(A) $a$ is even and $(p-a+1)/2 \leq j \leq (p-1)/2$  if $p$ is odd\\
(B) $a$ is even and $0 \leq j \leq (p-a-1)/2$ if $p$ is odd, or $(a,j)=(0,1/2)$ if $p=2$\\
(C) $a$ is odd and $0 \leq j \leq (a-1)/2$  if $p$ is odd, or $(a,j)=(1,0)$ if $p=2$\\
(D) $a$ is odd and $(a+1)/2 \leq j \leq (p-1)/2$  if $p$ is odd, or $(a,j)=(1,1)$ if $p=2$\\ 

\begin{Lemma}
Let $(a,j) \in \mathcal{P} $. 
The following holds. \\ \\
{\rm (i)} If $(a,j)$ satisfies {\rm (A)}, then $n(a,j)= (p-a-1)/2 +j$. \\
{\rm (ii)} If $(a,j)$ satisfies {\rm (B)}, then $n(a,j)= (p-a-1)/2 -j$. \\
{\rm (iii)} If $(a,j)$ satisfies {\rm (C)}, then $n(a,j)= (2p-a-1)/2 -j$. \\
{\rm (iv)} If $(a,j)$ satisfies {\rm (D)}, then $n(a,j)= j- (a+1)/2 $. \\ \\
On the right-hand side of each equality, $a$ and $j$ are regarded as the 
corresponding integers with 
$0 \leq a \leq p-1$ and $0 \leq j \leq (p-1)/2$ except for $j$ when $p=2$.
\end{Lemma}

\begin{proof}
It is clear if $p=2$, so we may assume that $p$ is odd. 
We have  
$$\prod_{i \in \mathbb{F}_p} \big(x-i(i+a+1)\big) = \Big( x+\big((a+1)/2\big)^2 \Big) 
\prod_{m \in \mathcal{S}-\{0\}} \Big( x+\big((a+1)/2\big)^2-m^2 \Big)^2$$
since $\varphi_{a,p}(x)= \psi \Big(x+\big((a+1)/2\big)^2\Big)$. Hence 
each factor $x-i(i+a+1)$ with $i \in \mathbb{F}_p$ on the left-hand side has 
the form $x+\big((a+1)/2\big)^2-t^2$ with $t \in \mathcal{S} $. Since 
\begin{eqnarray*}
\psi_j \Big(x+\big((a+1)/2\big)^2\Big)
& = & 2 \Big( x+\big((a+1)/2\big)^2 \Big) \Big( x+\big((a+1)/2\big)^2 +j^2 \Big) \\
&  & \times \prod_{m \in \mathcal{S} - \{0, j\} } \Big( x+\big((a+1)/2\big)^2-m^2 \Big)^2
\end{eqnarray*}
if $j \neq 0$, whereas  
$$\psi_j \Big(x+\big((a+1)/2\big)^2\Big)= \prod_{m \in \mathcal{S}-\{0\}} 
\Big( x+\big((a+1)/2\big)^2-m^2 \Big)^2$$
if $j=0$, it is only $x+\big((a+1)/2\big)^2 -j^2$ among the linear polynomials 
$x+\big((a+1)/2\big)^2 -t^2$ with $t \in \mathcal{S}  $ 
that does not appear as a factor of the polynomial 
$\psi_j \Big(x+\big((a+1)/2\big)^2\Big)$. Moreover, since 
$$x-n(n+a+1) \mid \psi_j \Big(x+\big((a+1)/2\big)^2\Big)$$
for any non-negative integer $n$ with $n < n(a,j)$ and since 
$$x-n(a,j) \big( n(a,j)+a+1 \big) \nmid  \psi_j \Big(x+\big((a+1)/2\big)^2 \Big)$$ 
by the definition of $n(a,j)$, the factor $x-n(a,j) \big( n(a,j)+a+1 \big)$ must be equal to 
$x+\big((a+1)/2\big)^2-j^2$. Therefore, we see that $n(a,j)$ is the smallest non-negative integer $n$ 
satisfying $-n(n+a+1)=\big((a+1)/2\big)^2 -j^2$ in $\mathbb{F}_p$. Note that this implies 
$\big( n(a,j) + (a+1)/2 \big)^2 = j^2$ in $\mathbb{F}_p$ and hence 
$n(a,j) = \pm j- (a+1)/2$ in $\mathbb{F}_p$. 

Suppose that the pair $(a,j)$ satisfies the condition (A). Then we see that 
$n(a,j) = (p-a-1)/2 +j $ or $(3p-a-1)/2 -j $ since 
$n(a,j) \equiv \pm j + (p-a-1)/2$ $({\rm mod}\ p)$ and $(p-a+1)/2 \leq j \leq (p-1)/2$. 
By the minimality of $n(a,j)$ we obtain $n(a,j) = (p-a-1)/2 +j $. 

Suppose that the pair $(a,j)$ satisfies the condition (B). Then we see that 
$n(a,j) = (p-a-1)/2 -j $ or $(p-a-1)/2 +j $ since 
$n(a,j) \equiv \pm j + (p-a-1)/2\ ({\rm mod}\ p)$ and $0 \leq j \leq (p-a-1)/2$. 
By the minimality of $n(a,j)$ we obtain $n(a,j) = (p-a-1)/2 -j $. 

Suppose that the pair $(a,j)$ satisfies the condition (C). Then we see that 
$n(a,j) = (2p-a-1)/2 -j $ or $(2p-a-1)/2 +j $ since 
$n(a,j) \equiv \pm j - (a+1)/2\ ({\rm mod}\ p)$ and $0 \leq j \leq (a-1)/2$. 
By the minimality of $n(a,j)$ we obtain $n(a,j) = (2p-a-1)/2 -j $. 

Suppose that the pair $(a,j)$ satisfies the condition (D). Then we see that 
$n(a,j) = j- (a+1)/2  $ or $(2p-a-1)/2 -j $ since 
$n(a,j) \equiv \pm j - (a+1)/2\ ({\rm mod}\ p)$ and $(a+1)/2 \leq j \leq (p-1)/2$. 
By the minimality of $n(a,j)$ we obtain $n(a,j) = j- (a+1)/2  $. 
\end{proof}

For a pair $(a,j) \in \mathcal{P}$, we set 
$\tilde{n}(a,j)=n(-a,j)$ if $p$ is odd, and 
$$\tilde{n}(0,1/2)=0,\ \ \tilde{n}(1,0)=0,\ \ \tilde{n}(1,1)=1$$
if $p=2$.

\begin{Lemma}
Let $(a,j) \in \mathcal{P} $. 
The following holds. 
\\ \\
{\rm (i)} If $(a,j)$ satisfies {\rm (A)}, then $\tilde{n}(a,j)= (-p+a-1)/2 +j$. \\
{\rm (ii)} If $(a,j)$ satisfies {\rm (B)}, then $\tilde{n}(a,j)= (p+a-1)/2 -j$. \\
{\rm (iii)} If $(a,j)$ satisfies {\rm (C)}, then $\tilde{n}(a,j)= (a-1)/2 -j$. \\
{\rm (iv)} If $(a,j)$ satisfies {\rm (D)}, then $\tilde{n}(a,j)= j+ (a-1)/2 $. \\ \\
On the right-hand side of each equality, $a$ and $j$ are regarded as the 
corresponding integers with 
$0 \leq a \leq p-1$ and $0 \leq j \leq (p-1)/2$ except for $j$ when $p=2$.
\end{Lemma}

\begin{proof}
It is clear if $p=2$, so we may assume that $p$ is odd. 
If $a=0$, the lemma holds by Lemma 4.6 (ii) since each $(0,j) \in \mathcal{P}$ satisfies (B). 
So we may assume  $a \neq 0$. Then  the element $-a \in \mathbb{F}_p$ corresponds to $p-a$ in the set 
of integers $\{ 0,1, \cdots , p-1 \}$. Set $a' = p-a$. 

Suppose that the pair $(a,j)$ satisfies the condition (A). Then $(a',j)$ satisfies (D) and so  we have 
$n(-a, j)=n(a',j)= j-(p-a+1)/2$ by Lemma 4.6 (iv). 

Suppose that the pair $(a,j)$ satisfies the condition (B). Then $(a',j)$ satisfies (C) and so  we have 
$n(-a, j)=n(a',j)= (p+a-1)/2 -j$ by Lemma 4.6 (iii). 

Suppose that the pair $(a,j)$ satisfies the condition (C). Then $(a',j)$ satisfies (B) and so  we have 
$n(-a, j)=n(a',j)= (a-1)/2 -j$ by Lemma 4.6 (ii). 

Suppose that the pair $(a,j)$ satisfies the condition (D). Then $(a',j)$ satisfies (A) and so  we have 
$n(-a, j)=n(a',j)= (a-1)/2 +j$ by Lemma 4.6 (i). 
\end{proof}

\begin{Lemma}
Let $(a,j) \in \mathcal{P}  $. Then the primitive idempotent $E(a,j)$ can be written as
$$E(a,j)= \mu_a \sum_{m=n(a,j)}^{p-1} c_m Y^m X^m = 
\mu_a \sum_{m=\tilde{n}(a,j)}^{p-1} \tilde{c}_m X^m Y^m$$
for some $c_m, \tilde{c}_m \in \mathbb{F}_p$ with $c_{n(a,j)} \neq 0$ and 
$\tilde{c}_{\tilde{n}(a,j)} \neq 0$. 
\end{Lemma}

\begin{proof}
It is clear when $p=2$, so we may assume that $p$ is odd. 
By the definition of $n(a,j)$ we can write 
$$\psi_{j}\Big(x+\big((a+1)/2\big)^2\Big) = 
\tau_{j} \Big(x+\big((a+1)/2\big)^2\Big) \cdot \varphi_{a, n(a,j)} (x) \eqno(*)$$
for some $\tau_j(x) \in \mathbb{F}_p [x]$.  Then  
\begin{eqnarray*}
E(a,j) & = &  \psi_{j} \Big(\mu_a YX+\big((a+1)/2\big)^2\Big) \cdot \mu_a \\
        & = &  \tau_j \Big(\mu_a YX +\big((a+1)/2\big)^2\Big) \cdot 
                 \varphi_{a,n(a,j)}(\mu_a YX) \cdot \mu_a \\
        & = &  \mu_a Y^{n(a,j)} X^{n(a,j)} \cdot  \tau_j \Big(\mu_a YX +\big((a+1)/2\big)^2\Big).
\end{eqnarray*}
Since 
\begin{eqnarray*}
\lefteqn{\mu_a Y^{n(a,j)} X^{n(a,j)} \cdot \mu_a YX} \\
& = & \mu_a Y^{n(a,j)+1} X^{n(a,j)+1} +n(a,j)\big(n(a,j)+a+1\big) \mu_a Y^{n(a,j)} X^{n(a,j)},
\end{eqnarray*}
we see that $E(a,j)$ can be written as $E(a,j)= \mu_a \sum_{m=n(a,j)}^{p-1} c_m Y^m X^m$ for some $c_m \in \mathbb{F}_p$ with $n(a,j) \leq m \leq p-1$. Note that 
$\tau_{j} \Big(x+\big((a+1)/2\big)^2\Big)$ is a product of some linear polynomials, but 
$x- n(a,j)\big(n(a,j)+a+1\big)$ does not 
appear as a factor of it. 
Thus  we conclude that $c_{n(a,j)} \neq 0$, and the first equality in the lemma is 
proved. 

Similarly, we obtain
$$E(a,j) = \mu_a X^{n(-a,j)} Y^{n(-a,j)} \cdot 
\tau_{j} \Big(\mu_a XY +\big((a-1)/2\big)^2\Big),$$
using the equality $(*)$ where $a$ is replaced by $-a$. Then a similar argument 
in the last paragraph shows the second equality in the lemma.   
\end{proof}

The projective indecomposable $\mathcal{U}_1$-module generated by $E(a,j)$ will be determined in Theorem 5.8  as the result for $r=1$ there.

\section{Primitive idempotents in $\mathcal{U}_r$}
To construct primitive idempotents in $\mathcal{U}_r$ for $r \geq 2$ we need the $k$-linear map 
${\rm Fr}' : \mathcal{U} \rightarrow \mathcal{U}$ which was introduced  in Section 2. Recall that the linear map ${\rm Fr}'$ is defined  by 
$$Y^{(m)} {H \choose n} X^{(m')} \mapsto 
Y^{(mp)} {H \choose np} X^{(m'p)}$$
for $m, m',n \in \mathbb{Z}_{\geq 0}$.

First we construct primitive idempotents in $\mathcal{U}_r^0$.  For $a \in \mathbb{Z}$, define an element 
$\mu_a^{(r)} \in \mathcal{U}_r^0$ as  
$$\mu_a^{(r)} = {H-a-1 \choose p^r-1}.$$
We have $\mu_a^{(1)} = \mu_a$, and $\mu_a^{(r)} = \mu_b^{(r)}$ if and only if 
$a \equiv b \ ({\rm mod}\ p^r)$. Then the following holds. 

\begin{Proposition}
{\rm (i)} Suppose that $r \geq 2$. Then for an integer $a = a_0 +  a'p$ with $0 \leq a_0 \leq p-1$ 
and $a' \in \mathbb{Z}$, we have $\mu_a^{(r)} = \mu_{a_0} {\rm Fr}' (\mu_{a'}^{(r-1)})$. 
In particular, $\mu_a^{(r)}$ is a $\mathcal{U}_r^0$-weight vector of $\mathcal{U}_r^0$-weight $a$. 
\\ \\
{\rm (ii)} The elements $\mu_a^{(r)}$ with $a \in \mathbb{Z} / p^r \mathbb{Z}$ are pairwise orthogonal 
primitive idempotents in $\mathcal{U}_r^0$ whose sum is $1 \in \mathcal{U}_r^0$. \\ \\
{\rm (iii)} $\mu_{a}^{(r)} X^{(m)} = X^{(m)} \mu_{a-2m}^{(r)}$ and 
$\mu_{a}^{(r)} Y^{(m)} = Y^{(m)} \mu_{a+2m}^{(r)}$ 
for $a \in \mathbb{Z}$ and $m \in \mathbb{Z}_{\geq 0}$. 
\end{Proposition}
\noindent For details, see 
4.7 and 4.8 in \cite{gros-kaneda15} if $p$ is odd, but  the proposition also holds 
for $p=2$. 

If a pair $(a,j) \in \mathcal{P}$ satisfies (A) or (C), set $s(a,j) =(p-a+1)/2$ if $p$ is odd and 
$a$ is even,  $s(a,j) =(p-a)/2$ if $p$ is odd and 
$a$ is odd, and $s(a,j)=1$ if $p=2$, regarding $a$ as the corresponding integer with 
$0 \leq a \leq p-1$.  

The following lemma will be used later. 

\begin{Lemma}
{\rm (i)} Let $a,b \in \mathbb{Z}$ with $0 \leq a \leq b \leq p-1$. Then 
$$\mu_a {\rm Fr}'(z_1) {\rm Fr}'(z_2) X^b = \mu_a {\rm Fr}' (z_1 z_2) X^b$$
for any $z_1, z_2 \in \mathcal{U}$. \\ \\
{\rm (ii)} Suppose that a pair $(a,j) \in  \mathcal{P} $ satisfies 
{\rm (A)} or {\rm (C)}.  Let $m,n \in \mathbb{Z}_{\geq 0}$ with $ n(a,j) \leq m \leq p-1$. 
Then $X^{(np)}$ and $Y^{(np)}$ commute with $\mu_a Y^m X^{m-s(a,j)}$, whereas 
$$\mu_a Y^m X^{m-s(a,j)} {H \choose np} = 
\bigg( {H \choose np} + {H \choose (n-1)p} \bigg) \mu_a Y^m X^{m-s(a,j)} $$ 
if we define ${H \choose t} =0$ for $t<0$. 
\end{Lemma}

\begin{proof}
(i) Without loss of generality we may assume that 
$z_j=Y^{(m_j)} {H \choose n_j} X^{(m_j')} $ 
with $m_j, m_j', n_j \in \mathbb{Z}_{\geq 0}$ and $j \in \{ 1,2 \}$. For simplicity we set 
$m = m_2$, $m'= m_1'$, 
${\bf Y} = Y^{(m_1)}$, ${\bf X} = X^{(m_2')}$, ${\bf H}_1 = {H \choose n_1}$ and 
${\bf H}_2 = {H \choose n_2}$. Then we can write $z_1= {\bf Y} {\bf H}_1 X^{(m')}$ 
and $z_2= Y^{(m)} {\bf H}_2 {\bf X}$. For each $l \in \mathbb{Z}_{\geq 0}$ there exist 
${\bf H}_{1,l}$, ${\bf H}_{2,l} \in \mathcal{U}^0$ such that 
${\bf H}_1 Y^{(l)} = Y^{(l)} {\bf H}_{1,l}$ and $X^{(l)} {\bf H}_2 = {\bf H}_{2,l} X^{(l)}$. 
Then we have 
\begin{eqnarray*}
\lefteqn{\mu_a  {\rm Fr}' (z_1 z_2) X^b} \\
& = & \mu_a {\rm Fr}' ({\bf Y} {\bf H}_1 X^{(m')} Y^{(m)} {\bf H}_2 {\bf X}) X^b \\
& = &  \mu_a {\rm Fr}' \bigg( {\bf Y} {\bf H}_1 \sum_{i=0}^{{\rm min}(m,m')} Y^{(m-i)} 
{H-m-m'+2i \choose i} X^{(m'-i)}  {\bf H}_2 {\bf X} \bigg) X^b \\
& = & \sum_{i=0}^{{\rm min}(m,m')}  \mu_a {\rm Fr}' \bigg( {\bf Y} Y^{(m-i)}  {\bf H}_{1, m-i}
{H-m-m'+2i \choose i}  {\bf H}_{2,m'-i} X^{(m'-i)}  {\bf X} \bigg) X^b.
\end{eqnarray*}
On the other hand, we have 
\begin{eqnarray*}
\lefteqn{\mu_a  {\rm Fr}' (z_1) {\rm Fr}'( z_2) X^b} \\
& = & \mu_a {\rm Fr}'^{-}({\bf Y}) {\rm Fr}'^{0}({\bf H}_1) {\rm Fr}'^{+}(X^{(m')}) 
{\rm Fr}'^{-}(Y^{(m)})  {\rm Fr}'^{0}({\bf H}_2) {\rm Fr}'^{+}({\bf X}) X^b \\
& = & \mu_a {\rm Fr}'^{-}({\bf Y}) {\rm Fr}'^{0}({\bf H}_1) X^{(m'p)} 
Y^{(mp)}  {\rm Fr}'^{0}({\bf H}_2) {\rm Fr}'^{+}({\bf X}) X^b \\
& = & \sum_{i=0}^{{\rm min}(mp,m'p)}  {\rm Fr}'^{-}({\bf Y}) {\rm Fr}'^{0}({\bf H}_1) 
\cdot \mu_a Y^{(mp-i)} {H-(m+m')p +2i \choose i} \\
&  & \times X^{(m'p-i)} {\rm Fr}'^{0}({\bf H}_2) {\rm Fr}'^{+}({\bf X}) X^b. 
\end{eqnarray*}
In this sum we consider the $i$-th summand. Choose a unique integer $i'$ satisfying 
$i \equiv i'\ ({\rm mod}\ p)$ and $0 \leq i' \leq p-1$. Then we can write 
$i = i' +  i''p$ for some $i'' \in \mathbb{Z}_{\geq 0}$. Suppose that $i' \neq 0$. We have 
\begin{eqnarray*}
\mu_a Y^{(mp-i)} {H-(m+m')p+2i \choose i} 
& = & \mu_{a}  {H+(m-m')p \choose i} Y^{(mp-i)}  \\
& = & \mu_{a}  \sum_{l=0}^{i} {(m-m')p \choose l} {H \choose i-l} Y^{(mp-i)} \\
& = & \mu_{a}  \sum_{l=0}^{i''} {(m-m')p \choose lp} {H \choose i-lp} Y^{(mp-i)} \\
& = & \mu_{a}  \sum_{l=0}^{i''} {(m-m')p \choose lp} 
{H \choose i'} {H \choose (i''-l)p} Y^{(mp-i)} \\
& = &  \mu_{a}  \sum_{l=0}^{i''} {(m-m')p \choose lp} 
{a \choose i'} {H \choose (i''-l)p} Y^{(mp-i)}. 
\end{eqnarray*}  
On the other hand, if we choose ${\bf H}_3 \in \mathcal{U}^0$ such that 
$X^{(m'p-i)} {\rm Fr}'^0({\bf H}_2) = {\bf H}_3 X^{(m'p-i)}$, then  
\begin{eqnarray*}
X^{(m'p-i)} {\rm Fr}'^0({\bf H}_2) {\rm Fr}'^{+} ({\bf X}) X^b 
& = & {\bf H}_3 X^{(m'p-i)} {\rm Fr}'^{+} ({\bf X}) X^b \\
& = & {\bf H}_3 {\rm Fr}'^{+} ({\bf X}) X^{(m'p-i)}  X^b \\
& = & b ! {(m'-i'')p+b-i' \choose b} {\bf H}_3  {\rm Fr}'^{+} ({\bf X}) X^{(m'p -i +b)}.
\end{eqnarray*} 
But either 
$a < i'$ or $b \geq i'$ holds since $a \leq b$, and so either 
$\mu_a Y^{(mp-i)} {H-(m+m')p+2i \choose i} $ or 
$X^{(m'p-i)} {\rm Fr}'^0({\bf H}_2) {\rm Fr}'^{+} ({\bf X}) X^b $ must be zero. 
Therefore, the $i$-th summand in the equality of $\mu_a  {\rm Fr}' (z_1) {\rm Fr}'( z_2) X^b$ 
is zero if $p$ does not divide $i$ and we have  
\begin{eqnarray*}
\lefteqn{\mu_a  {\rm Fr}' (z_1) {\rm Fr}'( z_2) X^b} \\
& = & \sum_{i=0}^{{\rm min}(mp,m'p)}  {\rm Fr}'^{-}({\bf Y}) {\rm Fr}'^{0}({\bf H}_1) 
\cdot \mu_a Y^{(mp-i)} {H-(m+m')p +2i \choose i} \\
&  & \times X^{(m'p-i)} {\rm Fr}'^{0}({\bf H}_2) {\rm Fr}'^{+}({\bf X}) X^b \\
& = & \sum_{i=0}^{{\rm min}(m,m')}  {\rm Fr}'^{-}({\bf Y}) {\rm Fr}'^{0}({\bf H}_1) 
\cdot \mu_a Y^{(mp-ip)} {H-(m+m'-2i)p  \choose ip} \\
&  & \times X^{(m'p-ip)} {\rm Fr}'^{0}({\bf H}_2) {\rm Fr}'^{+}({\bf X}) X^b.
\end{eqnarray*}
It is easy to see that 
$${H-(m+m'-2i)p \choose ip} 
 =  {\rm Fr}'^0 \bigg( {H-(m+m')+2i \choose i} \bigg)$$
by using Proposition 2.1 (v). Recall also that the restriction maps ${\rm Fr}'^{\geq 0}, {\rm Fr}'^{\leq 0},{\rm Fr}'^{0},{\rm Fr}'^{+}$ and ${\rm Fr}'^{-}$ are homomorphisms of $k$-algebras. Therefore, we obtain 
\begin{eqnarray*}
\lefteqn{\mu_a  {\rm Fr}' (z_1) {\rm Fr}'( z_2) X^b} \\
& = & \sum_{i=0}^{{\rm min}(m,m')}  \mu_a 
{\rm Fr}'^{-}({\bf Y}) {\rm Fr}'^{0}({\bf H}_1) 
{\rm Fr}'^{-}(Y^{(m-i)}) {\rm Fr}'^0 \bigg( {H-(m+m') +2i  \choose i} \bigg) \\
&  & \times {\rm Fr}'^{+} (X^{(m'-i)}) {\rm Fr}'^{0}({\bf H}_2) 
{\rm Fr}'^{+}({\bf X}) X^b \\
& = & \sum_{i=0}^{{\rm min}(m,m')}  \mu_a 
{\rm Fr}'^{\leq 0} \bigg( {\bf Y} {\bf H}_1 Y^{(m-i)} {H-(m+m') +2i  \choose i} \bigg) 
{\rm Fr}'^{\geq 0} (X^{(m'-i)} {\bf H}_2 {\bf X}) X^b \\
& = & \sum_{i=0}^{{\rm min}(m,m')}  \mu_a 
{\rm Fr}'^{\leq 0} \bigg( {\bf Y} Y^{(m-i)} {\bf H}_{1, m-i}  
{H-(m+m') +2i  \choose i} \bigg) \\ 
&  & \times {\rm Fr}'^{\geq 0} ({\bf H}_{2,m'-i} X^{(m'-i)}  {\bf X}) X^b \\
& = & \sum_{i=0}^{{\rm min}(m,m')}  \mu_a 
{\rm Fr}'^{-}({\bf Y} Y^{(m-i)})
{\rm Fr}'^{0} \bigg(  {\bf H}_{1, m-i}  
{H-(m+m') +2i  \choose i} \bigg) {\rm Fr}'^0 ({\bf H}_{2,m'-i} )\\ 
&  & \times {\rm Fr}'^{+} (X^{(m'-i)}  {\bf X}) X^b \\
& = & \sum_{i=0}^{{\rm min}(m,m')}  \mu_a 
{\rm Fr}'^{-}({\bf Y} Y^{(m-i)})
{\rm Fr}'^{0} \bigg(  {\bf H}_{1, m-i}  
{H-(m+m') +2i  \choose i}  {\bf H}_{2,m'-i} \bigg) \\ 
&  & \times {\rm Fr}'^{+} (X^{(m'-i)}  {\bf X}) X^b \\
& = & \sum_{i=0}^{{\rm min}(m,m')}  \mu_a 
{\rm Fr}'  \bigg( {\bf Y} Y^{(m-i)}  {\bf H}_{1, m-i}  
{H-(m+m') +2i  \choose i}  {\bf H}_{2,m'-i} X^{(m'-i)}  {\bf X} \bigg) X^b. \\
\end{eqnarray*}
The last sum is equal to $\mu_a  {\rm Fr}' (z_1  z_2) X^b$, and (i) is proved. 

(ii) It is easy for $p=2$, since $\mu_a Y^m X^{m-s(a,j)}=\mu_1 Y$ in this situation. 
So we may assume that $p$ is odd. Set $s=s(a,j)$. It is clear if $n=0$, so we may assume that $n>0$. 
We have 
\begin{eqnarray*}
X^{(np)} \mu_a Y^m X^{m-s} 
& = &  \mu_a X^{(np)} Y^m X^{m-s} \\
& = & m! \mu_a \sum_{i=0}^{m} Y^{(m-i)} {H -np-m+2i \choose i} X^{(np-i)} X^{m-s} \\
& = & m! \mu_a \sum_{i=0}^{m} {H -np+m \choose i} Y^{(m-i)}  X^{(np-i)} X^{m-s} \\
& = & m! \mu_a \sum_{i=0}^{m} {a -np+m \choose i} Y^{(m-i)}  X^{m-s} X^{(np-i)},  
\end{eqnarray*}
where the last equality follows from $i \leq p-1$. Consider a summand for $i \neq 0$ 
of this sum. Since $X^{m-s} X^{(np-i)} = (m-s) ! {m-s+np-i \choose np-i} X^{(m-s+np-i)}$,  
we must have $i>m-s$ if 
$X^{m-s} X^{(np-i)}  \neq 0$. On the other hand, we must have $a +m -p \geq i$ 
if ${a -np +m \choose i} \neq 0$ in $\mathbb{F}_p$ since 
$p \leq a+m \leq 2p-2$. But 
the inequalities $i > m-s$ and $a +m -p \geq i$ 
do not hold simultaneously since $(a+m-p)-(m-s) = a+s-p$ is not positive. Therefore, each  summand for $i \neq 0$ must be zero 
and we obtain $X^{(np)} \mu_a Y^m X^{m-s} = \mu_a Y^m X^{m-s} X^{(np)}$. 

In turn, we have 
$$\mu_a Y^m X^{m-s} Y^{(np)} = \sum_{i=0}^{m-s } (m-s)! \mu_a Y^m Y^{(np-i)} 
{H-np-m+s+2i \choose i} X^{(m-s-i)}.$$
However, on the right-hand side, only the summand for $i=0$ survives. 
Indeed, if $m>s$, we see that  
$$Y^m Y^{(np-i)} = m! {m-i+np \choose m} Y^{(m-i+np)} =0$$
for $1 \leq i \leq m-s$, since ${m-i+np \choose m} =0$ in $\mathbb{F}_p$.  
Therefore, we obtain  
\begin{eqnarray*}
\mu_a Y^m X^{m-s} Y^{(np)} 
& = & \mu_a Y^m Y^{(np)} X^{m-s} \\
& = & Y^{(np)} \mu_a Y^m X^{m-s}.
\end{eqnarray*}
Finally, observe that 
\begin{eqnarray*}
\mu_a Y^m X^{m-s} {H \choose np} 
& = & \mu_a {H+2s \choose np} Y^m X^{m-s} \\
& = & \mu_a \sum_{l=0}^{np} {2s \choose l} {H \choose np-l} Y^m X^{m-s}. 
\end{eqnarray*}
Since $2s \leq p-1$, we have 
\begin{eqnarray*}
 \mu_a \sum_{l=0}^{np} {2s \choose l} {H \choose np-l} 
& = & \mu_a \sum_{l=0}^{p-1} {2s \choose l} {H \choose np-l} \\
& = & \mu_a \bigg( {H \choose np} +
\sum_{l=1}^{p-1} {2s \choose l} {H \choose p-l} {H \choose (n-1)p} \bigg)   \\
& = & \mu_a \bigg( {H \choose np} +
\sum_{l=1}^{p-1} {2s \choose l} {a \choose p-l} {H \choose (n-1)p} \bigg).
\end{eqnarray*}
But it is easy to see that $\sum_{l=1}^{p-1} {2s \choose l} {a \choose p-l} =1$ 
in $\mathbb{F}_p$, and we obtain 
$$\mu_a Y^m X^{m-s} {H \choose np} 
=  \bigg( {H \choose np} + {H \choose (n-1)p} \bigg) \mu_a Y^m X^{m-s},$$ 
the proof is complete. 
\end{proof}

Now we shall construct primitive idempotents in $\mathcal{U}_r$. Suppose that $r \geq 2$. 
Recall from Lemma 4.8 that each primitive idempotent $E(a,j)$ in $\mathcal{U}_1$ can be written as 
$$E(a,j)= \mu_a \sum_{m=n(a,j)}^{p-1} c_m Y^m X^m = 
\mu_a \sum_{m=\tilde{n}(a,j)}^{p-1} \tilde{c}_m X^m Y^m$$
for some $c_m, \tilde{c}_m \in \mathbb{F}_p$ with $c_{n(a,j)} \neq 0$ and 
$\tilde{c}_{\tilde{n}(a,j)} \neq 0$. For each $z \in \mathcal{U}$, using the above notation 
we define an element 
$Z\big(z; (a,j)\big) \in \mathcal{U}$ as follows. \\ \\
$\cdot$ If the pair $(a,j)$ satisfies (A) or (C), then 
$$Z\big(z; (a,j)\big) = \mu_a \sum_{m=n(a,j)}^{p-1} c_m 
Y^m X^{m-s(a,j)} {\rm Fr}' (z) X^{s(a,j)}.$$
 \\ 
$\cdot$ If the pair $(a,j)$ satisfies (B) or (D), then 
$$Z\big(z; (a,j)\big) = {\rm Fr}' (z)  E(a,j).$$

\begin{Lemma}
For a pair $(a,j) \in \mathcal{P}$ and a nonzero element $z \in \mathcal{U}$, there exists a nonzero element 
$z' \in \mathcal{U}$ such that $Z(z; (a,j)) =  {\rm Fr}' (z')  E(a,j) 
= E(a,j)  {\rm Fr}' (z') $. Moreover, 
if $z \in \mathcal{A}$, then $Z(z; (a,j))$ also lies in $\mathcal{A}$. 
\end{Lemma}

\begin{proof}
Since $E(a,j) \in \mathcal{A}_1$, the commutativity of ${\rm Fr}' (z')$ and $E(a,j)$ follows immediately from 
Propositions 2.4 (i) and 2.5. So we have to show the first equality in the lemma. 
There is   nothing to do if the pair $(a,j)$ satisfies (B) or (D). So we may assume that 
$(a,j)$ satisfies (A) or (C). Set $s=s(a,j)$. It is enough to show the first equality for $z= Y^{(n_1)} {H \choose n_2} X^{(n_3)}$ 
with $n_1, n_2, n_3 \in \mathbb{Z}_{\geq 0}$. By Lemma 5.2 (ii)  we have 
\begin{eqnarray*}
\lefteqn{Z \bigg( Y^{(n_1)} {H \choose n_2} X^{(n_3)} ; (a,j) \bigg) } \\
& = & \mu_a \sum_{m=n(a,j)}^{p-1} c_m Y^m X^{m-s} 
Y^{(n_1p)} {H \choose n_2 p} X^{(n_3 p)} X^s \\
& = & Y^{(n_1p)} \bigg( {H \choose n_2 p} + {H \choose (n_2-1)p } \bigg) X^{(n_3 p)} 
\mu_a \sum_{m=n(a,j)}^{p-1} c_m Y^m X^{m} \\
& = & {\rm Fr}' \bigg( Y^{(n_1)} \bigg( {H \choose n_2 } + 
{H \choose n_2-1 } \bigg) X^{(n_3)} \bigg)  E(a,j), 
\end{eqnarray*}
as required. Moreover, if 
$ Y^{(n_1)} {H \choose n_2} X^{(n_3)}$ lies in $\mathcal{A}$ (i.e. $n_1= n_3$), clearly 
$Y^{(n_1 )}   \big(  {H \choose n_2 }  + 
 {H \choose n_2-1  } \big) X^{(n_3)}$ and 
$Z \big( Y^{(n_1)} {H \choose n_2} X^{(n_3)} ; (a,j) \big)$ 
also lie in $\mathcal{A}$. 
\end{proof}

\begin{Proposition}
{\rm (i)} For $z \in \mathcal{U}$ and 
$(a,j) \in \mathcal{P} $, we have $E(a,j) Z\big(z;(a,j)\big) = Z\big(z;(a,j)\big) = 
 Z\big(z;(a,j)\big)E(a,j).$ \\ \\ 
{\rm (ii)} For a pair $(a,j) \in \mathcal{P}$, the map 
$Z\big(-;(a,j)\big): \mathcal{U} \rightarrow \mathcal{U},\ z \mapsto Z\big(z;(a,j)\big)$ 
induces an injective  $k$-algebra homomorphism from 
$\mathcal{U}$ to $ E(a,j)\mathcal{U}E(a,j)$.  \\ \\
{\rm (iii)} Let $z_1, z_2 \in \mathcal{U}$ and 
$(a_1,j_1), (a_2,j_2) \in \mathcal{P}$. If 
$(a_1,j_1) \neq (a_2, j_2)$, then we have $Z\big(z_1 ;(a_1 ,j_1)\big) Z\big(z_2 ;(a_2,j_2)\big)= 0$. 
\\ \\
{\rm (iv)} Let $u$ be an element of the subalgebra of $\mathcal{U}$ 
generated by $X^{(p^i)}$ and $Y^{(p^i)}$ with 
$i \geq 1$. Then, for $z \in \mathcal{U}$, we have 
$u Z\big(z;(a,j)\big) =  Z\big({\rm Fr}(u)z;(a,j)\big).$
\end{Proposition}

\begin{proof}
(i) There exists $z' \in \mathcal{A}$ such that 
$Z\big(z; (a,j)\big) =  {\rm Fr}' (z')  E(a,j) =   E(a,j){\rm Fr}' (z') $ by Lemma 5.3. Since 
$E(a,j)$ is an idempotent, the claim follows.

(ii) By (i), clearly all $Z\big(z;(a,j)\big)$ lie in $E(a,j)\mathcal{U}E(a,j)$. The linearity of the 
map is also clear and the injectivity follows from Lemma 5.3 and 
Proposition 2.3. Note also that 
$Z\big(1;(a,j)\big) = E(a,j)$. So we only have to show that 
$Z\big(z_1 ;(a,j)\big) Z\big(z_2 ;(a,j)\big)= Z\big(z_1 z_2; (a,j)\big)$ for 
all $z_1,z_2 \in \mathcal{U}$. Suppose that 
the pair $(a,j)$ satisfies (B) or (D). Then 
\begin{eqnarray*}
Z\big(z_1 ;(a,j)\big) Z\big(z_2 ;(a,j)\big)
& = & {\rm Fr}' (z_1) E(a,j){\rm Fr}' (z_2) E(a,j) \\
& = & {\rm Fr}' (z_1) {\rm Fr}' (z_2) E(a,j) \\
& = & \sum_{m=\tilde{n}(a,j)}^{p-1} \tilde{c}_m \mu_a {\rm Fr}'(z_1) {\rm Fr}'(z_2) X^m Y^m \\
& = & \sum_{m=\tilde{n}(a,j)}^{p-1} \tilde{c}_m \mu_a {\rm Fr}'(z_1 z_2) X^m Y^m \ \ \ 
\big(\mbox{by Lemma 5.2 (i)} \big)\\
& = &  {\rm Fr}'(z_1 z_2) \sum_{m=\tilde{n}(a,j)}^{p-1} \tilde{c}_m \mu_a X^m Y^m \\
& = &  {\rm Fr}'(z_1 z_2) E(a,j) \\
& = & Z\big(z_1 z_2; (a,j)\big).
\end{eqnarray*}

\noindent Next, suppose that the pair $(a,j)$ satisfies (A) or (C). Set $s=s(a,j)$. 
Then 
\begin{eqnarray*}
\lefteqn{Z\big(z_1;(a,j)\big) Z\big(z_2;(a,j)\big)} \\
& = & \bigg( \mu_a \sum_{m=n(a,j)}^{p-1} c_m 
Y^m X^{m-s} {\rm Fr}' (z_1) X^{s} \bigg) \cdot 
 \bigg( \mu_a \sum_{m'=n(a,j)}^{p-1} c_{m'} 
Y^{m'} X^{m'-s} {\rm Fr}' (z_2) X^{s} \bigg) \\
& = &  \mu_a \sum_{m=n(a,j)}^{p-1} c_m 
Y^m X^{m-s}  \cdot \mu_{a+2s} {\rm Fr}' (z_1)
  \sum_{m'=n(a,j)}^{p-1} c_{m'} 
 X^{s} Y^{m'} X^{m'-s} {\rm Fr}' (z_2) X^{s} \\
& = &  \mu_a \sum_{m=n(a,j)}^{p-1} c_m 
Y^m X^{m-s} \sum_{m'=n(a,j)}^{p-1} c_{m'} 
 X^{s} Y^{m'} X^{m'-s} \cdot \mu_{a+2s} {\rm Fr}' (z_1)
   {\rm Fr}' (z_2) X^{s}.
\end{eqnarray*}
Note that 
$a+2s=0$ or $1$ in $\mathbb{F}_p$, and hence that 
$ \mu_{a+2s} {\rm Fr}' (z_1)
   {\rm Fr}' (z_2) X^{s} = 
 \mu_{a+2s} {\rm Fr}' (z_1 z_2) X^{s}$ 
by Lemma 5.2 (i). Therefore, the last term in the above equalities is equal to 
\begin{eqnarray*}
\lefteqn{E(a,j) \sum_{m'=n(a,j)}^{p-1} c_{m'} 
 Y^{m'} X^{m'-s} \cdot \mu_{a+2s} {\rm Fr}' (z_1 z_2) X^{s}} \\
& = & E(a,j) \cdot \mu_{a}\sum_{m'=n(a,j)}^{p-1} c_{m'} 
 Y^{m'} X^{m'-s}   {\rm Fr}' (z_1 z_2) X^{s} \\
& = & E(a,j) Z\big(z_1 z_2; (a,j)\big) \\
& = &  Z\big(z_1 z_2; (a,j)\big),
\end{eqnarray*}
where the last equality follows from (i), and the claim follows.  

(iii) We know that $E(a_1, j_1) E(a_2, j_2)=0$ if $(a_1, j_1) \neq (a_2, j_2)$. 
Then the claim easily follows from Lemma 5.3. 

(iv) It is enough to show the equality for $u=X^{(p^i)}$ and $u=Y^{(p^i)}$ with $i \geq 1$. 
Using Lemma 5.2 we have 
\begin{eqnarray*}
X^{(p^i)}Z\big(z ;(a,j)\big)
& = & {\rm Fr}' (X^{(p^{i-1})}){\rm Fr}' (z) E(a,j) \\
& = & \sum_{m=\tilde{n}(a,j)}^{p-1} \tilde{c}_m \mu_a {\rm Fr}'(X^{(p^{i-1})}) {\rm Fr}'(z) X^m Y^m \\
& = & \sum_{m=\tilde{n}(a,j)}^{p-1} \tilde{c}_m \mu_a {\rm Fr}'(X^{(p^{i-1})} z) X^m Y^m \\
& = &  {\rm Fr}'( X^{(p^{i-1})}z) \sum_{m=\tilde{n}(a,j)}^{p-1} \tilde{c}_m \mu_a X^m Y^m \\
& = &  {\rm Fr}'( X^{(p^{i-1})}z) E(a,j) \\
& = & Z\big({\rm Fr}(X^{(p^i)}) z; (a,j)\big)
\end{eqnarray*}
if the pair $(a,j)$ satisfies (B) or (D), and 
\begin{eqnarray*}
X^{(p^i)}Z\big(z ;(a,j)\big)
& = & X^{(p^i)} \bigg( \mu_a^2 \sum_{m=n(a,j)}^{p-1} c_m 
Y^m X^{m-s} {\rm Fr}' (z) X^{s} \bigg)  \\
& = & \mu_a^2 \sum_{m=n(a,j)}^{p-1} c_m 
Y^m X^{m-s} X^{(p^i)}{\rm Fr}' (z) X^{s}    \\
& = &  \mu_a \sum_{m=n(a,j)}^{p-1} c_m 
Y^m X^{m-s} \mu_{a+2s}{\rm Fr}'(X^{(p^{i-1})}){\rm Fr}' (z) X^{s}    \\ 
& = &  \mu_a \sum_{m=n(a,j)}^{p-1} c_m 
Y^m X^{m-s} \mu_{a+2s}{\rm Fr}'(X^{(p^{i-1})}z) X^{s}   \\
& = &  \mu_a^2 \sum_{m=n(a,j)}^{p-1} c_m 
Y^m X^{m-s} {\rm Fr}'(X^{(p^{i-1})}z) X^{s}   \\
& = & Z\big({\rm Fr}(X^{(p^{i})})z ;(a,j)\big)
\end{eqnarray*}
if the pair $(a,j)$ satisfies (A) or (C), where $s=s(a,j)$. 

Similarly we obtain $Y^{(p^i)}Z\big(z ;(a,j)\big)=Z\big({\rm Fr}(Y^{(p^{i})})z ;(a,j)\big)$ 
for $i \geq 1$, the proof is complete.
\end{proof}

For $n$ pairs $(a_i,j_i) \in \mathcal{P},\ 0 \leq i \leq n-1$,  define 
$E\big((a_0, \cdots , a_{n-1}), (j_0, \cdots , j_{n-1}) \big)$ inductively as  
$$E\big((a_0),(j_0)\big)= E(a_0, j_0)$$
and 
$$E\big((a_0, \cdots , a_{n-1}), (j_0, \cdots , j_{n-1}) \big)= 
Z \Big( E\big((a_1, \cdots , a_{n-1}), (j_1, \cdots , j_{n-1}) \big) ; (a_0, j_0) \Big) $$
for $n \geq 2$. Note that all $E\big((a_0, \cdots , a_{n-1}), (j_0, \cdots , j_{n-1}) \big)$ 
lie in $\mathcal{A}_n$. 

\begin{Proposition}
For $r$ pairs $(a_i,j_i) \in \mathcal{P},\ 0 \leq i \leq r-1$, the following holds. \\ \\
{\rm (i)} We have 
$$
\mu_{\sum_{i=0}^{r-1} b_i p^i}^{(r)} 
E\big((a_0, \cdots , a_{r-1}), (j_0, \cdots , j_{r-1}) \big)
=  E\big((a_0, \cdots , a_{r-1}), (j_0, \cdots , j_{r-1}) \big), $$
where
$$b_i = 
\left\{ \begin{array}{ll} 
{a_i-p} & {\mbox{if $(a_i, j_i)$ satisfies {\rm (A)} or {\rm (C)},}} \\
{a_i} & {\mbox{if $(a_i, j_i)$ satisfies {\rm (B)} or {\rm (D)}}} 
\end{array}, \right.$$
regarding each $a_i$ as the corresponding integer 
with $0 \leq a_i \leq p-1$. In particular, 
$E\big((a_0, \cdots , a_{r-1}), (j_0, \cdots , j_{r-1}) \big)$ is a $\mathcal{U}_r^{0}$-weight vector 
of $\mathcal{U}_r^0$-weight $\sum_{i=0}^{r-1} b_i p^i$. \\ \\
{\rm (ii)} Suppose that $r \geq 2$. Then  
$$\sum_{(a_i,j_i) \in \mathcal{P},\ 1 \leq i \leq r-1} 
E\big((a_0, \cdots , a_{r-1}),(j_0, \cdots , j_{r-1})\big)= E(a_0, j_0).$$ 
\ \\
{\rm (iii)}  The elements $E\big((a_0, \cdots , a_{r-1}), (j_0, \cdots , j_{r-1}) \big)$  with 
 $(a_i,j_i) \in \mathcal{P},\ 0 \leq i \leq r-1$ are pairwise 
orthogonal  primitive idempotents in $\mathcal{U}_r$ whose sum is $1$. 
\end{Proposition}

\begin{proof} We use induction on $r$. 
(i) and (iii) for $r=1$ follow immediately from the definition of $E(a_0,j_0)$ and Proposition 4.5. 
Suppose that $r \geq 2$. 

(i) We see that $\mu^{(r)}_{\sum_{i=0}^{r-1} b_i p^i} E(a,j) = 
Z \big( \mu^{(r-1)}_{\sum_{i=1}^{r-1} b_i p^{i-1}} ; (a_0,j_0)\big)$. Indeed, by Proposition 5.1 
(i) and (iii) we have 
\begin{eqnarray*}
Z \big( \mu^{(r-1)}_{\sum_{i=1}^{r-1} b_i p^{i-1}} ; (a_0,j_0)\big) 
& = & {\rm Fr}'(\mu^{(r-1)}_{\sum_{i=1}^{r-1} b_i p^{i-1}}) E(a_0,j_0) \\
& = & {\rm Fr}'(\mu^{(r-1)}_{\sum_{i=1}^{r-1} b_i p^{i-1}}) 
\mu_{a_0} E(a_0,j_0) \\
& = & \mu_{a_0+\sum_{i=1}^{r-1} b_i p^{i}}^{(r)} E(a_0, j_0) \\
& = & \mu_{\sum_{i=0}^{r-1} b_i p^{i}}^{(r)} E(a_0, j_0)
\end{eqnarray*}
if the pair $(a_0,j_0)$ satisfies (B) or (D), and 
\begin{eqnarray*}
Z \big( \mu^{(r-1)}_{\sum_{i=1}^{r-1} b_i p^{i-1}} ; (a_0,j_0)\big) 
& = & \mu_{a_0}^2 \sum_{m=n(a_0,j_0)}^{p-1} c_m Y^m X^{m-s} 
{\rm Fr}'(\mu^{(r-1)}_{\sum_{i=1}^{r-1} b_i p^{i-1}}) X^s \\
& = & \mu_{a_0} \sum_{m=n(a_0,j_0)}^{p-1} c_m Y^m X^{m-s} \mu_{a_0+2s}
{\rm Fr}'(\mu^{(r-1)}_{\sum_{i=1}^{r-1} b_i p^{i-1}}) X^s \\
& = & \mu_{a_0} \sum_{m=n(a_0,j_0)}^{p-1} c_m Y^m X^{m-s} 
\mu^{(r)}_{a_0+2s-p +\sum_{i=1}^{r-1} b_i p^{i}} X^s \\
& = & \mu^{(r)}_{a_0-p +\sum_{i=1}^{r-1} b_i p^{i}} \cdot 
\mu_{a_0} \sum_{m=n(a_0,j_0)}^{p-1} c_m Y^m X^{m} \\
& = & \mu_{\sum_{i=0}^{r-1} b_i p^{i}}^{(r)} E(a_0, j_0)
\end{eqnarray*}
if the pair $(a_0,j_0)$ satisfies (A) or (C), where $s=s(a_0,j_0)$.  
Then using Proposition 5.4 (i) and (ii) 
we have 
\begin{eqnarray*}
\lefteqn{\mu_{\sum_{i=0}^{r-1} b_i p^i}^{(r)} 
E\big((a_0, \cdots , a_{r-1}), (j_0, \cdots , j_{r-1}) \big)} \\
& = & \mu_{\sum_{i=0}^{r-1} b_i p^i}^{(r)} 
Z \Big( E\big((a_1, \cdots , a_{r-1}), (j_1, \cdots , j_{r-1}) \big); (a_0,j_0) \Big) \\
& = & \mu_{\sum_{i=0}^{r-1} b_i p^i}^{(r)} E(a_0,j_0)
Z \Big( E\big((a_1, \cdots , a_{r-1}), (j_1, \cdots , j_{r-1}) \big); (a_0,j_0) \Big) \\
& = & Z \big( \mu^{(r-1)}_{\sum_{i=1}^{r-1} b_i p^{i-1}} ; (a_0,j_0)\big) 
Z \Big( E\big((a_1, \cdots , a_{r-1}), (j_1, \cdots , j_{r-1}) \big); (a_0,j_0) \Big) \\
& = & Z \Big( \mu^{(r-1)}_{\sum_{i=1}^{r-1} b_i p^{i-1}} E\big((a_1, \cdots , a_{r-1}), (j_1, \cdots , j_{r-1}) \big); (a_0,j_0) \Big) .
\end{eqnarray*}
By induction, we have 
$$\mu^{(r-1)}_{\sum_{i=1}^{r-1} b_i p^{i-1}} E\big((a_1, \cdots , a_{r-1}), (j_1, \cdots , j_{r-1}) \big) = E\big((a_1, \cdots , a_{r-1}), (j_1, \cdots , j_{r-1}) \big),$$
and the claim follows.

(ii) By Proposition 5.4 (ii) and induction, we have 
\begin{eqnarray*}
\lefteqn{\sum_{(a_i,j_i) \in \mathcal{P},\ 1 \leq i \leq r-1} 
E\big((a_0, \cdots , a_{r-1}),(j_0, \cdots , j_{r-1})\big)} \\
& = & \sum_{(a_i,j_i) \in \mathcal{P},\ 1 \leq i \leq r-1} 
Z \Big( E\big((a_1, \cdots , a_{r-1}),(j_1, \cdots , j_{r-1})\big) ; (a_0,j_0) \Big) \\
& = & 
Z \bigg( \sum_{(a_i,j_i) \in \mathcal{P},\ 1 \leq i \leq r-1} 
E\big((a_1, \cdots , a_{r-1}),(j_1, \cdots , j_{r-1})\big) ; (a_0,j_0) \bigg) \\
& = & Z \big( 1 ; (a_0,j_0)\big) \\
& = & E(a_0,j_0),
\end{eqnarray*}
as desired.

(iii)  If $(a_0, j_0) \neq (a_0', j_0')$, we have by 
Proposition 5.4 (iii) that 
$$E\big((a_0, \cdots , a_{r-1}),(j_0, \cdots , j_{r-1})\big) 
E\big((a_0', \cdots , a_{r-1}'),(j_0', \cdots , j_{r-1}')\big)=0,$$
so we may assume  that $(a_0, j_0) = (a_0', j_0')$. Then 
\begin{eqnarray*}
\lefteqn{E\big((a_0, \cdots , a_{r-1}),(j_0, \cdots , j_{r-1})\big) 
E\big((a_0', \cdots , a_{r-1}'),(j_0', \cdots , j_{r-1}')\big)} \\
& = & Z \Big( E\big((a_1, \cdots , a_{r-1}),(j_1, \cdots , j_{r-1})\big) ; (a_0,j_0) \Big) \\
&  & \times Z \Big( E\big((a_1', \cdots , a_{r-1}'),(j_1', \cdots , j_{r-1}')\big) ; (a_0,j_0) \Big) \\
& = & Z \Big( E\big((a_1, \cdots , a_{r-1}),(j_1, \cdots , j_{r-1})\big) 
E\big((a_1', \cdots , a_{r-1}'),(j_1', \cdots , j_{r-1}')\big) ; (a_0,j_0) \Big)
\end{eqnarray*}
by Proposition 5.4 (ii). By induction, this is equal to 
$$Z \Big( E\big((a_1, \cdots , a_{r-1}),(j_1, \cdots , j_{r-1})\big) ; (a_0,j_0) \Big) 
= E\big((a_0, \cdots , a_{r-1}),(j_0, \cdots , j_{r-1})\big)$$
if $(a_i,j_i)=(a_i',j_i')$ for all $i$ with $1 \leq i \leq r-1$, and to
$$Z\big(0; (a_0, j_0)\big) = 0$$
otherwise. Therefore, we conclude that all 
$E\big((a_0, \cdots , a_{r-1}),(j_0, \cdots , j_{r-1})\big)$ 
are pairwise orthogonal idempotents. The fact that a sum of these idempotents is 1 
follows from (ii) and Proposition 4.5. It remains to show that these idempotents are primitive. Consider an integer $\lambda$ with $0 \leq \lambda \leq p^r-1$ and its 
$p$-adic expansion $\lambda = \sum_{i=0}^{r-1} \lambda_i p^i$.  
It follows from Steinberg's tensor product theorem that 
$${\rm dim}_k L(\lambda) = \prod_{i=0}^{r-1} {\rm dim}_k L(\lambda_i) = 
\prod_{i=0}^{r-1} (\lambda_i + 1). $$
Therefore, the number of the summands in a decomposition of $1 \in \mathcal{U}_r$ 
into pairwise orthogonal primitive idempotents is 
\begin{eqnarray*}
\sum_{\lambda=0}^{p^r-1} {\rm dim}_k L(\lambda) 
& = & \sum_{(\lambda_0, \cdots , \lambda_{r-1}) \in \{ 0,1, \cdots , p-1 \}^r } 
(\lambda_0+1) \cdots (\lambda_{r-1}+1) \\
& = & \prod_{i=0}^{r-1} \bigg(\sum_{\lambda_i =0}^{p-1} (\lambda_i +1)\bigg) \\
& = & \prod_{i=0}^{r-1} \big(p(p+1)/2\big) \\
& = & \big(p(p+1)/2\big)^r.
\end{eqnarray*}
On the other hand, the number of the elements  
$E\big((a_0, \cdots , a_{r-1}),(j_0, \cdots , j_{r-1})\big)$ 
is $\big(p(p+1)/2\big)^r$ as well. Therefore, these idempotents must be primitive. 
\end{proof}

\noindent {\bf Remark.} Note that any idempotent in $\mathcal{U}$ must have degree $0$, 
namely, it must lie in $\mathcal{A}$. Since the subalgebra $\mathcal{A}$ (hence $\mathcal{A}_r$) is commutative, (iii) gives a unique decomposition of $1$ into 
a sum of primitive idempotents in $\mathcal{U}_r$ 
(see \cite[ch. 1. Theorem 4.6]{nagao-tsushimabook}). \\ 

We can also give primitive idempotents in $\mathcal{U}_{r,r'}$ ($r'>r$).

\begin{Proposition}
For $r' > r$, the elements 
$$E\big((a_0, \cdots , a_{r-1}), (j_0, \cdots , j_{r-1}) \big)
{\rm Fr}'^r \big(\mu_{a'}^{(r'-r)}\big)$$  
with 
 $(a_i,j_i) \in \mathcal{P},\ 0 \leq i \leq r-1$ and $0 \leq a' \leq p^{r'-r}-1$ are pairwise 
orthogonal  primitive idempotents in $\mathcal{U}_{r,r'}$ whose sum is $1$. 
\end{Proposition}

\begin{proof}
By Proposition 5.5 (iii), 
all  $E\big((a_0, \cdots , a_{r-1}), (j_0, \cdots , j_{r-1}) \big)$ are 
pairwise orthogonal primitive idempotents in  $\mathcal{A}_{r}$ whose sum is $1$. 
On the other hand, by Proposition 5.1 (ii), all 
${\rm Fr}'^r \big(\mu_{a'}^{(r'-r)}\big)$ (with $0 \leq a' \leq p^{r'-r}-1$) are pairwise orthogonal primitive idempotents in  the $k$-algebra ${\rm Fr}'^r(\mathcal{U}_{r'-r}^0)$ whose sum is $1$. 
Note that the multiplication map 
$\mathcal{A}_{r} \otimes {\rm Fr}'^r(\mathcal{U}_{r'-r}^0) \rightarrow \mathcal{A}_{r,r'}$ is not 
only a $k$-linear isomorphism (see Proposition 2.3) 
but also a $k$-algebra isomorphism since $\mathcal{A}$ is commutative. Therefore, all  
$$E\big((a_0, \cdots , a_{r-1}), (j_0, \cdots , j_{r-1}) \big)
{\rm Fr}'^r \big(\mu_{a'}^{(r'-r)}\big)$$  
are pairwise 
orthogonal  primitive idempotents in $\mathcal{A}_{r,r'}$ whose sum is $1$. 
But these are primitive also in $\mathcal{U}_{r,r'}$ since any 
idempotent in $\mathcal{U}$ lies in $\mathcal{A}$. 
\end{proof}

The following lemma is used to determine the PIMs generated by the primitive idempotents. 

\begin{Lemma}
Let $(a_i, j_i) \in \mathcal{P},\ 0 \leq i \leq r-1$, and let $t$ is the largest integer $n$ with 
$X^{(n)} E\big((a_0, \cdots , a_{r-1}), (j_0, \cdots , j_{r-1}) \big) \neq 0$ and $0 \leq n \leq p^r-1$. Then  $t = \sum_{i=0}^{r-1} \big(p-1- \tilde{n}(a_i, j_i)\big) p^i$. 
\end{Lemma}

\begin{proof}
If $r=1$, the result follows from Lemma 4.8. Suppose that $r \geq 2$. We use 
 induction on $r$. Let $n$ be an integer with $0 \leq n \leq p^r-1$ and 
write $n=n_0 +  n'p$ for some integers $n_0, n'$ with $0 \leq n_0 \leq p-1$ and 
$0 \leq n' \leq p^{r-1}-1$.  Then 
\begin{eqnarray*}
\lefteqn{X^{(n)} E\big((a_0, \cdots , a_{r-1}), (j_0, \cdots , j_{r-1}) \big)} \\
& = & X^{(n_0)} X^{(n'p)} 
Z\Big( E\big((a_1, \cdots , a_{r-1}), (j_1, \cdots , j_{r-1}) \big);  (a_0,j_0)\Big)  \\
& = & X^{(n_0)}  
Z\Big( X^{(n')}E\big((a_1, \cdots , a_{r-1}), (j_1, \cdots , j_{r-1}) \big);  (a_0,j_0)\Big)
\end{eqnarray*}
by Proposition 5.4 (iv). By induction the largest integer $n'$ with  
$$X^{(n')}  E\big((a_1, \cdots , a_{r-1}), (j_1, \cdots , j_{r-1}) \big) \neq 0$$ 
and $0 \leq n' \leq p^{r-1}-1$ is $\sum_{i=1}^{r-1} \big(p-1-\tilde{n}(a_i,j_i)\big)p^{i-1}$. 
Then, by Lemma 5.3, there is a nonzero element $z' \in \mathcal{U}$ 
such that 
$$X^{(n_0)}  
Z\Big( X^{(n')}E\big((a_1, \cdots , a_{r-1}), (j_1, \cdots , j_{r-1}) \big);  (a_0,j_0)\Big)
= X^{(n_0)}E(a_0, j_0) {\rm Fr}'(z'). $$
Hence the largest integer $n_0$ where this term does not vanish is 
$p-1-\tilde{n}(a_0,j_0)$ by Lemma 4.8. Thus we obtain 
\begin{eqnarray*} 
t & = & \big(p-1-\tilde{n}(a_0,j_0)\big) +  p \sum_{i=1}^{r-1} 
\big(p-1-\tilde{n}(a_i,j_i)\big)p^{i-1} \\
  & = &  \sum_{i=0}^{r-1} \big(p-1- \tilde{n}(a_i, j_i)\big) p^i, 
\end{eqnarray*}
and the lemma follows. 
\end{proof}

Now we describe the main result.

\begin{Theorem}
For $0 \leq i \leq r-1$, let $(a_i,j_i) \in \mathcal{P}$, and set 
$$\beta_i = 
\left\{ \begin{array}{ll} 
{p-2j_i-1} & {\mbox{if $(a_i, j_i)$ satisfies {\rm (B)} or {\rm (C)},}} \\
{2j_i-1} & {\mbox{if $(a_i, j_i)$ satisfies {\rm (A)} or {\rm (D)}}} 
\end{array} \right.$$
and 
$$b_i = 
\left\{ \begin{array}{ll} 
{a_i-p} & {\mbox{if $(a_i, j_i)$ satisfies {\rm (A)} or {\rm (C)},}} \\
{a_i} & {\mbox{if $(a_i, j_i)$ satisfies {\rm (B)} or {\rm (D)}}} 
\end{array} \right.$$
regarding $a_i$ and $j_i$ as the corresponding integers with 
$0 \leq a_i \leq p-1$ and $0 \leq j_i \leq (p-1)/2$ except for $j_i$ when $p=2$. 
Moreover, let $a' \in \mathbb{Z}$ with  
$0 \leq a' \leq p^{r'-r}-1$.  Then the element 
$$E\big((a_0, \cdots , a_{r-1}), (j_0, \cdots , j_{r-1}) \big)
{\rm Fr}'^r \big(\mu_{a'}^{(r'-r)}\big)$$ generates a $\mathcal{U}_{r,r'}$-module 
isomorphic to 
$\widetilde{Q}_r \big( \sum_{i=0}^{r-1} \beta_i p^i  + a' p^r\big)$ if  
$\sum_{i=0}^{r-1} b_i p^i \geq 0$, and to 
$\widetilde{Q}_r \big( \sum_{i=0}^{r-1} \beta_i p^i  + (a'+1) p^r\big)$ if  
$\sum_{i=0}^{r-1} b_i p^i < 0$.  
In particular,  the element 
$E\big((a_0, \cdots , a_{r-1}), (j_0, \cdots , j_{r-1}) \big)$ generates a $\mathcal{U}_{r}$-module 
isomorphic to 
$Q_r \big( \sum_{i=0}^{r-1} \beta_i p^i \big)$. 
\end{Theorem}

\begin{proof}
Since $E\big((a_0, \cdots , a_{r-1}), (j_0, \cdots , j_{r-1}) \big)
{\rm Fr}'^r \big(\mu_{a'}^{(r'-r)}\big)$ is a primitive idempotent in 
$\mathcal{U}_{r,r'}$ by Proposition 5.6, it generates a projective indecomposable $\mathcal{U}_{r,r'}$-module isomorphic to $\widetilde{Q}_r(\lambda'+\lambda'' p^r)$ 
for certain integers $\lambda'$ and $\lambda''$ with $0 \leq \lambda' \leq p^{r}-1$ 
and  $0 \leq \lambda'' \leq p^{r'-r}-1$. 
The element $$E\big((a_0, \cdots , a_{r-1}), (j_0, \cdots , j_{r-1}) \big)
{\rm Fr}'^r \big(\mu_{a'}^{(r'-r)}\big)$$ has $\mathcal{U}_{r'}^0$-weight 
$\sum_{i=0}^{r-1} b_i p^i  + a' p^r$ if $\sum_{i=0}^{r-1} b_i p^i  \geq 0$, 
and $\sum_{i=0}^{r-1} b_i p^i  +p^r + a' p^r$ if 
$\sum_{i=0}^{r-1} b_i p^i  < 0$, since 
\begin{eqnarray*}
\lefteqn{E\big((a_0, \cdots , a_{r-1}), (j_0, \cdots , j_{r-1}) \big)
{\rm Fr}'^r \big(\mu_{a'}^{(r'-r)}\big)} \\
& = & E\big((a_0, \cdots , a_{r-1}), (j_0, \cdots , j_{r-1}) \big)
\mu_{\sum_{i=0}^{r-1} b_i p^i }^{(r)}{\rm Fr}'^r \big(\mu_{a'}^{(r'-r)}\big) \\
& = & 
\left\{ 
\begin{array}{ll}
{\mu_{\sum_{i=0}^{r-1} b_i p^i +a' p^r}^{(r)}E\big((a_0, \cdots , a_{r-1}), (j_0, \cdots , j_{r-1}) \big)} & {\mbox{if $\sum_{i=0}^{r-1} b_i p^i  \geq 0$},} \\
&  \\
{\mu_{\sum_{i=0}^{r-1} b_i p^i +p^r+a' p^r}^{(r)}E\big((a_0, \cdots , a_{r-1}), (j_0, \cdots , j_{r-1}) \big)} & {\mbox{if $\sum_{i=0}^{r-1} b_i p^i  < 0$}} 
\end{array}.
\right.
\end{eqnarray*}
By Lemma 4.7,  we see  that 
\begin{eqnarray*}
b_i +2\big(p-1-\tilde{n}(a_i, j_i)\big)
& = &  
 \left\{ \begin{array}{ll} 
{p-1+2j_i} & {\mbox{if $(a_i, j_i)$ satisfies {\rm (B)} or {\rm (C)},}} \\
{2p-1-2j_i} & {\mbox{if $(a_i, j_i)$ satisfies {\rm (A)} or {\rm (D)}}} 
\end{array} \right. \\
& = & 2(p-1) - \beta_i.
\end{eqnarray*}
Moreover,  by Lemma 5.7, the largest integer $t$ satisfying $0 \leq t \leq p^r-1$ and 
$X^{(t)} E\big((a_0, \cdots , a_{r-1}), (j_0, \cdots , j_{r-1}) \big) \neq 0$
is $\sum_{i=0}^{r-1} \big(p-1-\tilde{n}(a_i,j_i)\big) p^i$. 

Suppose that $\sum_{i=0}^{r-1} b_i p^i  \geq 0$.  Then we have  
\begin{eqnarray*}
\sum_{i=0}^{r-1} b_i p^i +2t 
& = & \sum_{i=0}^{r-1} \Big( b_i +2 \big(p-1-\tilde{n}(a_i , j_i)\big) \Big)p^i \\
& = &  \sum_{i=0}^{r-1} \big(2(p-1) - \beta_i \big)p^i  \\
& \leq  & 2p^r-2, 
\end{eqnarray*}
and hence $\lambda'' =a'$ and
\begin{eqnarray*}
\lambda' 
& = & 2p^r-2 - \bigg(\sum_{i=0}^{r-1} b_i p^i  +2t \bigg) \\
& = & 2p^r-2 -  \sum_{i=0}^{r-1} \big(2(p-1) - \beta_i \big)p^i  \\
& = & \sum_{i=0}^{r-1} \beta_i p^i
\end{eqnarray*}
by Proposition 3.1 (i). 

Finally, suppose   that $\sum_{i=0}^{r-1} b_i p^i < 0$. Then we have 
$$\bigg( \sum_{i=0}^{r-1} b_i p^i +p^r \bigg) +2t 
 =   \sum_{i=0}^{r-1} \big(2(p-1) - \beta_i \big)p^i +p^r > 2p^r-2.  
$$
If $a' \neq p^{r'-r}-1$, then $\lambda''= a'+1$ and 
\begin{eqnarray*}
\lambda' 
& = & 3p^r-2 - \bigg( \sum_{i=0}^{r-1} b_i p^i + p^r +2t \bigg) \\
& = & 3p^r-2 -  \sum_{i=0}^{r-1} \big(2(p-1) - \beta_i\big)p^i - p^r \\
& = & \sum_{i=0}^{r-1} \beta_i p^i 
\end{eqnarray*}
by Proposition 3.1 (ii). If $a' = p^{r'-r}-1$, then $\lambda''= 0$ and 
$\lambda' = \sum_{i=0}^{r-1} \beta_i p^i$ by Proposition 3.1 (iii), but 
$\widetilde{Q}_r(\sum_{i=0}^{r-1} \beta_i p^i) \cong 
\widetilde{Q}_r(\sum_{i=0}^{r-1} \beta_i p^i+ (a'+1)p^{r})$ as $\mathcal{U}_{r,r'}$-modules since 
$$\sum_{i=0}^{r-1} \beta_i p^i \equiv \sum_{i=0}^{r-1} \beta_i p^i +p^{r'} =
\sum_{i=0}^{r-1} \beta_i p^i +(a'+1)p^{r}\ ({\rm mod}\ p^{r'}). $$
Thus the first statement is proved. The second statement is clear. 
\end{proof}
\ \\
{\large {\bf Acknowledgment}} \\

The author would like to thank the referee for carefully reading the first and the 
second versions of this paper and giving  
helpful  comments.
      %<-------------------

% Your bilbigraphy           %<-------------------

\end{document}